\documentclass[12pt]{article}
\usepackage{amsmath,amsfonts,amsthm,amssymb,mathrsfs,latexsym,paralist,xy}
\usepackage{tikz}
\usepackage{soul}

\newtheorem{thm}{Theorem}[section]
\newtheorem{prop}[thm]{Proposition}
\newtheorem{cor}[thm]{Corollary}
\newtheorem{lemma}[thm]{Lemma}
\theoremstyle{remark}
\newtheorem{rmk}[thm]{Remark}
\newtheorem{example}[thm]{Example}
\theoremstyle{definition}
\newtheorem{defn}[thm]{Definition}

\DeclareMathOperator{\Ad}{Ad}

\newcommand{\bi}{\begin{itemize}}
\newcommand{\ei}{\end{itemize}}
\newcommand{\be}{\begin{enumerate}}
\newcommand{\ee}{\end{enumerate}}

\newcommand{\T}{\mathbb{T}}
\renewcommand{\H}{\mathcal{H}}

\newcommand{\F}{\mathcal{F}}

\newcommand{\R}{\mathbb{R}}
\newcommand{\N}{\mathbb{N}}
\newcommand{\Z}{\mathbb{Z}}

\providecommand{\keywords}[1]{{\textit{Keywords and phrases:}} #1}
\providecommand{\classification}[1]{{\textit{2010 Mathematics Subject Classification:}} #1}

\def\IoIIdimdots(#1/#2/#3,#4){\node at (#1,#4) {$.$};\node at (#2,#4) {$.$};\node at (#3,#4) {$.$};}
\def\IIoIIdimdots(#1,#2/#3/#4){\node at (#1,#2) {$.$};\node at (#1,#3) {$.$};\node at (#1,#4) {$.$};}

\def\IoIIIdimdots(#1/#2/#3,#4,#5){\node at (#1,#4,#5) {$.$};\node at (#2,#4,#5) {$.$};\node at (#3,#4,#5) {$.$};}
\def\IIoIIIdimdots(#1,#2/#3/#4,#5){\node at (#1,#2,#5) {$.$};\node at (#1,#3,#5) {$.$};\node at (#1,#4,#5) {$.$};}
\def\IIIoIIIdimdots(#1,#2,#3/#4/#5){\node at (#1,#2,#3) {$.$};\node at (#1,#2,#4) {$.$};\node at (#1,#2,#5) {$.$};}

\usepackage[margin=1in]{geometry}
\AtEndDocument{\bigskip{\small%
  \textsc{Carla Farsi, Judith Packer : Department of Mathematics, University of Colorado at Boulder, Boulder, CO 80309-0395, USA.} \par
  \textit{E-mail address}: \texttt{carla.farsi@colorado.edu, packer@euclid.colorado.edu} \par
  \addvspace{\medskipamount}
  \textsc{Elizabeth Gillaspy : Department of Mathematics, University of Montana, 32 Campus Drive \#0864, Missoula, MT 59812-0864, USA.}\par
  \textit{E-mail address}: \texttt{elizabeth.gillaspy@mso.umt.edu}\par
  \addvspace{\medskipamount}
  \textsc{Palle Jorgensen : Department of Mathematics, 14 MLH, University of Iowa, Iowa City, IA 52242-1419, USA.}\par
  \textit{E-mail address}: \texttt{palle-jorgensen@uiowa.edu}\par
   \addvspace{\medskipamount}
  \textsc{Sooran Kang : College of General Education, Chung-Ang University, 84 Heukseok-ro, Dongjak-gu, Seoul, Republic of Korea.} \par
  \textit{E-mail address}, \texttt{sooran09@cau.ac.kr}
}}

\begin{document}

\title{Representations of higher-rank graph $C^*$-algebras associated to $\Lambda$-semibranching function systems}
\author{Carla Farsi, Elizabeth Gillaspy, Palle Jorgensen,  Sooran Kang, and Judith Packer}
\date{\today}
\maketitle

\begin{abstract} 
In this paper, we discuss a method of constructing separable  representations of the  $C^*$-algebras associated to strongly connected row-finite $k$-graphs $\Lambda$.   We begin by giving an alternative characterization of the $\Lambda$-semibranching function systems introduced in an earlier paper, with an eye towards constructing such representations that are faithful. Our new characterization allows us to more easily check that examples satisfy certain necessary and sufficient conditions.  We present a variety of new examples relying on this characterization. We then use some of these methods and a direct limit procedure to construct a faithful separable representation for any row-finite source-free $k$-graph. 

\end{abstract}

\classification{46L05, 46L55, 46K10}

\keywords{$C^*$-algebras, representations, higher-rank graphs, $k$-graphs, $\Lambda$-semibranching function systems, Markov measures, Perron--Frobenius.}

\tableofcontents

\section{Introduction}

In  \cite{KP}, Kumjian and Pask introduced higher-rank graphs $\Lambda$ -- also known as $k$-graphs -- and their $C^*$-algebras $C^*(\Lambda)$  as
generalizations of the Cuntz and Cuntz--Krieger $C^*$-algebras associated to directed graphs (cf.~\cite{Cuntz, Cuntz-Krieger, enomoto-watatani,
kpr}).  The $C^*$-algebras of higher-rank graphs  are closely linked with orbit equivalence for shift spaces \cite{carlsen-ruiz-sims} and with symbolic dynamics more generally \cite{pask-raeburn-weaver,skalski-zacharias-houston, pask-sierakowski-sims}, as well as with fractals and self-similar structures \cite{FGJKP1, FGJKP2}. {More links between higher-rank graphs and symbolic dynamics can be seen via \cite{bezuglyi-k-m, bezuglyi-four-auth} and the references cited therein.} 
{Higher-rank graphs have also provided crucial examples \cite{ruiz-sims-sorensen} for Elliott's program \cite{elliott-AF, kirchberg-phillips,  tikuisis-white-winter} to classify $C^*$-algebras by $K$-theoretic invariants.} 

Despite this ubiquity of $k$-graph $C^*$-algebras, representations of $C^*(\Lambda)$ on separable Hilbert spaces are almost nonexistent in the literature. This motivated us to undertake 
the present detailed study of separable representations of $k$-graph $C^*$-algebras and their unitary equivalence classes.   
One of the few examples of separable representations of $C^*(\Lambda)$ was identified in \cite{FGKP}, using the notion of $\Lambda$-semibranching function systems introduced in that paper. These $\Lambda$-semibranching function systems generalize to the $k$-graph setting the semibranching function systems for Cuntz--Krieger algebras which were studied by K. Kawamura \cite{kawamura}, M. Marcolli and A. Paolucci \cite{MP}, and S. Bezuglyi and P. Jorgensen \cite{bezuglyi-jorgensen}.  Semibranching function systems, and iterated function systems more generally, also have applications to automata theory, as established in \cite{charlier}.

In  this paper, the representations associated to $\Lambda$-semibranching function systems, which are called the $\Lambda$-semibranching representations, form our jumping-off point. (See Definition~\ref{def-lambda-SBFS-1}.)
We begin in Section \ref{sec-fund-mater} with an introduction to higher-rank graphs and their $C^*$-algebras, followed by a review of the $\Lambda$-semibranching function systems introduced in \cite{FGKP}.
 We also present several results related to the Carath\'eodory/Kolmogorov Extension Theorem which we use repeatedly throughout this work.

Our first main result is Theorem~\ref{thm:SBFS-edge-defn} in Section \ref{sec:SBFS-revisited} which provides an alternative characterization of a $\Lambda$-semibranching function system, which is easier to check in examples.  We then use Theorem~\ref{thm:SBFS-edge-defn} to describe how to construct a $\Lambda$-semibranching function system on a finite $k$-graph $\Lambda$ when $\Lambda$ is given as a product graph of a $k_1$-graph $\Lambda_1$ and a $k_2$-graph $\Lambda_2$, $\Lambda=\Lambda_1\times \Lambda_2$: see Proposition~\ref{prop:product-graph-SBFS}.
Next we present a variety of examples of $\Lambda$-semibranching function systems on  measure spaces $(X, \mu)$ in   Section~\ref{sec:examples_a} (where $X$ is a Lebesgue measure space) and Section~\ref{sec:examples_gen_measure} (where $X = \Lambda^\infty$ is the infinite path space of the higher-rank graph). 
 Through careful computations of the Radon--Nikodym derivatives associated to these $\Lambda$-semibranching function systems, we analyze the relationship between their associated representations and the standard $\Lambda$-semibranching representation on $L^2(\Lambda^\infty, M)$ which was introduced in Proposition~3.4 and Theorem~3.5 of \cite{FGKP}.\footnote{The measure $M$ was introduced in Definition 8.1 of \cite{aHLRS3}.} % {\color{purple} (Also see Remark~\ref{rmk:defn-standard-repn})}.
In particular, for several examples of finite 2-graphs $\Lambda$, we construct product measures on $\Lambda^\infty$ which give rise to $\Lambda$-semibranching representations of $C^*(\Lambda)$ in Proposition~\ref{prop:example_exonevtwoe}. Moreover, for $x \in (0,1),$ we construct  Markov measures $\mu_x$ on $\Lambda_{2N}^\infty$ for a  family of 2-graphs $\{ \Lambda_{2N}\}$, such that for $x \not= 1/2$, $\mu_x$  is mutually singular to the Perron--Frobenius measure $M$ given in \cite{aHLRS3}. (See Proposition~\ref{prop-RN-of-Markov-2-example} and Proposition~\ref{prop:markov-2N}). Furthermore, for $x\ne x'$ we have that $\mu_x$ and $\mu_{x'}$ are mutually singular.

 In Section~\ref{sec:faithful-rep},  we move on to our second main result, which is the  construction of a faithful separable representation of the $C^*$-algebra associated to any row-finite, source-free $k$-graph in Theorem \ref{pr:sep-faith}.  By using direct limits and some techniques introduced in \cite{davidson-power-yang-dilation},  we are able to generalize Theorem  3.6 of \cite{FGKP}, which identified such a faithful separable representation   under the hypotheses that the $k$-graph was finite, strongly connected and aperiodic.
 The representation of Theorem \ref{pr:sep-faith} extends a simpler construction  given in Proposition~\ref{prop:sc-faithful} that can be used in the case where the graph is finite. The representation of Proposition~\ref{prop:sc-faithful} is initially defined on an inductive limit Hilbert space, rather than $L^2(X, \mu)$, and we show in Proposition \ref{prop:sc-faithful-SBFS} that this representation can be viewed as a $\Lambda$-semibranching representation.

While we were in the process of writing up the results presented below, 
D. Gon\c{c}alves, H.  Li,  and D. Royer posted a manuscript \cite{GLR} on the arXiv in which they introduce a definition of $\Lambda$-branching systems for more general $k$-graphs, called finitely aligned $k$-graphs.    

While there is some overlap between their work and ours, especially concerning the case of  $k$-graphs with one vertex, our work reducing the definition of $\Lambda$-semibranching function systems to a study of the elementary edges is independent of theirs, and our construction of the faithful separable representations of the higher-rank graph $C^*$-algebras is completely new.  We also hope that our focus in this paper on concrete examples of representations of finite higher-rank graph $C^*$-algebras will inspire more researchers to join us in studying these fascinating objects.

\subsection*{Acknowledgments}   	 
The authors would like to thank  Alex Kumjian and Daniel Gon\c calves  for helpful discussions.  E.G.~was partially supported by   the Deutsches Forschungsgemeinshaft via the SFB 878 ``Groups, Geometry, and Actions'' of the Universit\"at M\"unster. 	 
S.K.~was supported by Basic Science Research Program through the National Research Foundation of Korea (NRF) funded by the Ministry of Education (\#2017R1D2A1B03034697).
C.F. and J.P. were partially supported by two individual grants
from the Simons Foundation (C.F. \#523991; J.P. \#316981).
P.J. thanks his colleagues in the Math Department at the University of Colorado, for making a week-long visit there possible, for support, and for kind hospitality.

C.F. also thanks IMPAN for hospitality during her visits to IMPAN, Warsaw, Poland, where some of this work was carried out (grant \#3542/H2020/2016/2). This paper was partially supported by the grant H2020-MSCA-RISE-2015-691246-QUANTUM DYNAMICS.
\section{Foundational material}
\label{sec-fund-mater}

\subsection{Higher-rank graphs}

 We will now describe in detail higher-rank graphs and their $C^*$-algebras. 
Let $\N=\{0,1,2,\dots\}$ denote the monoid of natural numbers under addition, and let $k\in \N$ with $k\ge 1$. We write $e_1,\dots e_k$ for the standard basis vectors of $\N^k$, where $e_i$ is the vector of $\N^k$ with $1$ in the $i$-th position and $0$ everywhere else.

\begin{defn} \cite[Definition 1.1]{KP}
\label{def-higher-rank-gr}
A \emph{higher-rank graph} or \emph{$k$-graph} is a countable small category $\Lambda$ 
 with a degree functor $d:\Lambda\to \N^k$ satisfying the \emph{factorization property}: for any morphism $\lambda\in\Lambda$ and any $m, n \in \N^k$ such that  $d(\lambda)=m+n \in \N^k$,  there exist unique morphisms $\mu,\nu\in\Lambda$ such that $\lambda=\mu\nu$ and $d(\mu)=m$, $d(\nu)=n$. 
\end{defn}

%Readers who are new to the study of higher-rank graphs may wish to review the examples presented in Sections \ref{sec:examples_a} and \ref{sec:examples_gen_measure} below, before reading further.  The diagrams included with these examples, and the factorization rules described there, will give the reader more geometric, analytic, and combinatorial insight into the factorization property and its consequences.

%\st{ Readers not familiar with technical points in this definition (higher-rank graphs) may find illustrations, geometric, analytic, and combinatorial, contained in a list of concrete and detailed examples that we include in our analysis inside sections   below. Of direct relevance to the notions from Definition, see especially the figures contained in these two sections. }

%When discussing $k$-graphs, we use the arrows-only picture of category theory; thus, objects in $\Lambda$ are identified with identity morphisms, and the notation $ \lambda \in \Lambda $ means $\lambda$ is a morphism in $\Lambda$.

We often regard $k$-graphs as a generalization of directed graphs, so we call morphisms $\lambda\in\Lambda$ \emph{paths} in $\Lambda$, and the objects (identity morphisms) are often called \emph{vertices}. For $n\in\N^k$ and vertices $v,w$ of $\Lambda$, we write
\begin{equation}
\label{eq:Lambda-n}
\Lambda^n:=\{\lambda\in\Lambda\,:\, d(\lambda)=n\}\
\end{equation}
With this notation, note that $\Lambda^0$ is the set of objects (vertices) of $\Lambda$. Occasionally, we call elements of $\Lambda^{e_i}$ (for any $i$) \emph{edges}.  
We write $r,s:\Lambda\to \Lambda^0$ for the range and source maps in $\Lambda$ respectively, and 
\[v\Lambda w:=\{\lambda\in\Lambda\,:\, r(\lambda)=v,\;s(\lambda)=w\}.\]
  Combining this with Equation \eqref{eq:Lambda-n} results in abbreviations such as 
\[ v\Lambda^n:= \{ \lambda \in \Lambda: r(\lambda) = v, \ d(\lambda) = n\}\]
which we will use throughout the paper.

For $m,n\in\N^k$, we write $m\vee n$ for the coordinatewise maximum of $m$ and $n$. Given  $\lambda,\eta\in \Lambda$, we write
\begin{equation}\label{eq:lambda_min}
\Lambda^{\operatorname{min}}(\lambda,\eta):=\{(\alpha,\beta)\in\Lambda\times\Lambda\,:\, \lambda\alpha=\eta\beta,\; d(\lambda\alpha)=d(\lambda)\vee d(\eta)\}.
\end{equation}
%Then we denote the set of \emph{minimal common extensions} of $\lambda,\eta \in \Lambda$ by
%\[
%\operatorname{MCE}(\lambda,\eta):=\{\lambda\alpha\,:\, (\alpha,\beta)\in \Lambda^{\operatorname{min}}(\lambda,\eta)\}
%=\{\eta\beta\,:\, (\alpha,\beta)\in \Lambda^{\operatorname{min}}(\lambda,\eta)\}.
%\]
If $k=1$, then $\Lambda^{\operatorname{min}}(\lambda, \eta)$ will have at most one element; this need not be true in a $k$-graph if $k > 1$.

We say that a $k$-graph $\Lambda$ is \emph{finite} if $\Lambda^n$ is a finite set for all $n\in\N^k$ and say that $\Lambda$  \emph{has no sources} or \emph{is source-free} if $v\Lambda^n\ne \emptyset$ for all $v\in\Lambda^0$ and $n\in\N^k$. It is well known that this is equivalent to the condition that $v\Lambda^{e_i}\ne \emptyset$ for all $v\in \Lambda$ and all basis vectors $e_i$ of $\N^k$. 
We say that $\Lambda$ is \emph{row-finite} if $|v\Lambda^n| < \infty$ for all $v\in\Lambda^0$ and $n\in\N^k$, and we are mostly interested in finite (or row-finite) $k$-graphs in this paper; in fact all of our examples are finite $k$-graphs. 

We often visualize a $k$-graph as a (quotient of a) $k$-colored directed graph via the equivalence relation induced by the factorization rules.
To be precise, for each $1 \leq i \leq k$, we can define the \emph{$i$th vertex matrix} $A_i \in M_{\Lambda^0}(\N)$ by $A_i(v,w) = |v\Lambda^{e_i} w|$.
 Observe that the factorization rules imply that $A_iA_j=A_jA_i$ for $1\le i,j\le k$.  Indeed, given a pair of composable edges $f_1 \in v\Lambda^{e_i} z, f_2 \in z \Lambda^{e_j} w$, the factorization rule implies that since $d(f_1 f_2) = e_i + e_j = e_j + e_i$, the morphism $f_1 f_2 \in \Lambda$ can also be described uniquely as 
\[ f_1 f_2 = g_2 g_1 \;\;\text{ where } g_2 \in v\Lambda^{e_j} , \ g_1 \in  \Lambda^{e_i} w.\] 

We now describe two fundamental examples of higher-rank graphs which were first mentioned in \cite{KP}.

\begin{example}
\begin{itemize}
\item[(a)] For any directed graph $E$, let $\Lambda_E$ be the category of its finite paths. Then $\Lambda_E$ is a 1-graph with the degree functor $d:\Lambda_E\to \N$ which takes a finite path $\eta$ to its length $|\eta|$ (the number of edges making up $\eta$).
\item[(b)] For $k\ge 1$, let $\Omega_k$ be the small category with
\[
\operatorname{Obj}(\Omega_k)=\N^k,\quad \text{and}\quad \operatorname{Mor}(\Omega_k)=\{(p,q)\in \N^k\times \N^k\,:\, p\le q\}.
\]
Again, we can also view elements of $\text{Obj}(\Omega_k)$ as identity morphisms, via the map $\text{Obj}(\Omega_k) \ni p \mapsto (p, p) \in \text{Mor}(\Omega_k)$.
The range and source maps $r,s:\operatorname{Mor}(\Omega_k)\to \operatorname{Obj}(\Omega_k)$ are given by $r(p,q)=p$ and $s(p,q)=q$. If we define $d:\Omega_k\to \N^k$ by $d(p,q)=q-p$, then one can check that $\Omega_k$ is a $k$-graph with degree functor $d$.
\end{itemize}

\end{example}

\begin{defn}[\cite{KP} Definitions 2.1]
\label{def:infinite-path}
Let $\Lambda$ be a $k$-graph. An \emph{infinite path} in $\Lambda$ is a $k$-graph morphism (degree-preserving functor) $x:\Omega_k\to \Lambda$, and we write $\Lambda^\infty$ for the set of infinite paths in $\Lambda$. 
Since $\Omega_k$ has a terminal object (namely $0 \in\N^k$) but no initial object, we think of our infinite paths as having a range $r(x) : = x(0)$ but no source.
For each $m\in \N^k$, we have a shift map $\sigma^m:\Lambda^\infty \to \Lambda^\infty$ given by
\begin{equation}\label{eq:shift-map}
\sigma^m(x)(p,q)=x(p+m,q+m)
\end{equation}
for $x\in\Lambda^\infty$ and $(p,q)\in\Omega_k$.

We say that a $k$-graph $\Lambda$ is \emph{aperiodic} if for each $v\in\Lambda^0$, there exists $x\in v\Lambda^\infty$ such that for all $m\ne n\in\N^k$ we have $\sigma^m(x)\ne \sigma^n(x)$.
\end{defn}

It is well-known that the collection of cylinder sets 
\[
Z(\lambda)=\{x\in\Lambda^\infty\,:\, x(0,d(\lambda))=\lambda\},
\]
for $\lambda \in \Lambda$, form a compact open basis for a locally compact Hausdorff topology on $\Lambda^\infty$, under  reasonable hypotheses on $\Lambda$ (in particular, when $\Lambda$ is row-finite: see Section 2 of \cite{KP}). If a $k$-graph $\Lambda$ is  finite, then $\Lambda^\infty$ is compact in this topology. In fact, for a finite $k$-graph $\Lambda$, the proof of Lemma 4.1 from \cite{FGKP} establishes that the topology on $\Lambda^\infty$ (and hence the Borel $\sigma$-algebra $\mathcal{B}_o(\Lambda^\infty)$) is generated by the ``square'' cylinder sets
\[ \{ Z(\lambda): d(\lambda) = (n, \ldots, n) \text{ for some } n \in \N\}:\]
given any cylinder set $Z(\nu)$ with $d(\nu) \leq (n, \ldots, n)$, let
\[ I = \{ \lambda_i \in \Lambda: d(\nu \lambda_i)  = (n, \ldots, n) \}.\]
Then $Z(\lambda) = \bigsqcup_{\lambda_i \in I} Z(\nu \lambda_i)$ is a disjoint union of square cylinder sets.

According to Proposition 8.1 of \cite{aHLRS3}, for many finite higher-rank graphs there is a unique Borel probability measure $M$ on $\Lambda^\infty$ satisfying a certain self-similarity condition. 

\begin{defn}
We say that a $k$-graph is \emph{strongly connected} if, for all $v,w\in\Lambda^0$, $v\Lambda w\ne \emptyset$.
\end{defn}

If a $k$-graph $\Lambda$ is finite and strongly connected with vertex matrices $A_1,\dots A_k\in M_{\Lambda^0}(\N)$, then Proposition~3.1 of \cite{aHLRS3} implies that there is a unique positive vector $\kappa^\Lambda\in (0,\infty)^{\Lambda^0}$ such that $\sum_{v\in\Lambda^0}\kappa_v^{\Lambda}=1$ and for all $1\le i\le k$,
\[
A_i \kappa^{\Lambda}=\rho_i\, \kappa^{\Lambda},
\]
where $\rho_i$ denotes the spectral radius of $A_i$. The vector $\kappa^{\Lambda}$ is called the \emph{(unimodular) Perron--Frobenius eigenvector} of $\Lambda$. Then the measure $M$ on $\Lambda^\infty$ is given by
\begin{equation}
M(Z(\lambda))=(\rho(\Lambda))^{-d(\lambda)}\kappa^{\Lambda}_{s(\lambda)}\quad\text{for}\;\; \lambda\in\Lambda,
\label{eq:M}
\end{equation}
where $\rho(\Lambda)=(\rho_1,\dots \rho_k)$ and $(\rho(\Lambda))^n=\rho_1^{n_1}\dots \rho_k^{n_k}$ for $n=(n_1,\dots n_k)\in \Z^k$. We call the measure $M$ the \emph{Perron--Frobenius measure} on $\Lambda^\infty$.  Proposition 8.1 of \cite{aHLRS3} establishes that if $\mu$ is a Borel probability measure on $\Lambda^\infty$ such that 
\[
\mu(Z(\lambda))=\rho(\Lambda)^{-d(\lambda)}\mu (Z(s(\lambda)))\quad\text{for all}\;\; \lambda\in\Lambda,
\]
then $\mu = M$.

Now we introduce the $C^*$-algebra associated to a $k$-graph $\Lambda$. Here we only consider row-finite $k$-graphs with no sources. 
\begin{defn}\label{def:kgraph-algebra}
Let $\Lambda$ be a row-finite $k$-graph with no sources. A \emph{Cuntz--Krieger $\Lambda$-family} is a collection  $\{t_\lambda:\lambda\in\Lambda\}$ of partial isometries in a $C^*$-algebra satisfying
\begin{itemize}
\item[(CK1)] $\{t_v\,:\, v\in\Lambda^0\}$ is a family of mutually orthogonal projections,
\item[(CK2)] $t_\lambda t_\eta=t_{\lambda\eta}$ if $s(\lambda)=r(\eta)$,
\item[(CK3)] $t^*_\lambda t_\lambda=t_{s(\lambda)}$ for all $\lambda\in\Lambda$,
\item[(CK4)] for all $v\in\Lambda$ and $n\in\N^k$, we have
\[
t_v=\sum_{\lambda\in v\Lambda^n} t_\lambda t^*_\lambda.
\]
\end{itemize}
The Cuntz--Krieger $C^*$-algebra $C^*(\Lambda)$ associated to $\Lambda$ is the universal $C^*$-algebra generated by a Cuntz--Krieger $\Lambda$-family.
\end{defn}
One can show that 
\[
C^*(\Lambda)=\overline{\operatorname{span}}\{t_\alpha t^*_\beta\,:\, \alpha,\beta\in\Lambda,\; s(\alpha)=s(\beta)\}.
\]
Also, (CK4) implies that for all $\lambda, \eta\in \Lambda$, we have
\begin{equation}\label{eq:CK4-2}
t_\lambda^* t_\eta=\sum_{(\alpha,\beta)\in \Lambda^{\operatorname{min}}(\lambda,\eta)} t_\alpha t^*_\beta. 
\end{equation}
The universal property implies that the $C^*$-algebra $C^*(\Lambda)$ carries a strongly continuous action  $\gamma$ of the $k$-torus $\T^k$, called the \emph{gauge action}, which is given by
\[
\gamma_z(t_\lambda)=z^{d(\lambda)}t_\lambda,
\]
where $z^n=\prod_{i=1}^k z_i^{n_i}$ for $z=(z_1,\dots, z_k)\in \T^k$ and $n = (n_1, \ldots, n_k) \in\N^k$. Note that we only  discuss the gauge action in Section \ref{sec:faithful-rep}.

\subsection{$\Lambda$-semibranching function systems and their representations}

In \cite{FGKP}, separable representations of $C^*(\Lambda)$ were constructed by using $\Lambda$-semibranching function systems on measure spaces. A $\Lambda$-semibranching function system is a generalization of the semibranching function systems studied by Marcolli and Paolucci in \cite{MP}.   Here we review basic definitions and introduce the standard example of  a $\Lambda$-semibranching function system on $(\Lambda^\infty,M)$ and its associated representation: see Example~\ref{example:SBFS-M}.

\begin{defn}
\label{def-1-brach-system}\cite[Definition~2.1]{MP}\label{defn:sbfs}
Let $(X,\mu)$ be a measure space. Suppose that, for each $1\le i\le N$, we have a measurable map $\sigma_i:D_i\to X$, for some measurable subsets $D_i\subset X$. The family $\{\sigma_i\}_{i=1}^N$ is a \emph{semibranching function system} if the following holds:
\begin{itemize}
\item[(a)] Setting $R_i = \sigma_i(D_i),$ we have 
\[
\mu(X\setminus \cup_i R_i)=0,\quad\quad\mu(R_i\cap R_j)=0\;\;\text{for $i\ne j$}.
\]
\item[(b)] For each $i$, the Radon--Nikodym derivative
\[
\Phi_{\sigma_i}=\frac{d(\mu\circ\sigma_i)}{d\mu}
\]
satisfies $\Phi_{\sigma_i}>0$, $\mu$-almost everywhere on $D_i$.
\end{itemize}
A measurable map $\sigma:X\to X$ is called a \emph{coding map} for the family $\{\sigma_i\}_{i=1}^N$ if $\sigma\circ\sigma_i(x)=x$ for all $x\in D_i$.
\end{defn}

\begin{defn}\cite[Definition~3.2]{FGKP}
\label{def-lambda-SBFS-1}
Let $\Lambda$ be a finite $k$-graph and let $(X, \mu)$ be a measure space.  A \emph{$\Lambda$-semibranching function system} on $(X, \mu)$ is a collection $\{D_\lambda\}_{\lambda \in \Lambda}$ of measurable subsets of $X$, together with a family of prefixing maps $\{\tau_\lambda: D_\lambda \to X\}_{\lambda \in \Lambda}$, and a family of coding maps $\{\tau^m: X \to X\}_{m \in \N^k}$, such that
\begin{itemize}
\item[(a)] For each $m \in \N^k$, the family $\{\tau_\lambda: d(\lambda) = m\}$ is a semibranching function system, with coding map $\tau^m$.
\item[(b)] If $ v \in \Lambda^0$, then  $\tau_v = id$,  and $\mu(D_v) > 0$.
\item[(c)] Let $R_\lambda = \tau_\lambda( D_\lambda)$. For each $\lambda \in \Lambda, \nu \in s(\lambda)\Lambda$, we have $R_\nu \subseteq D_\lambda$ (up to a set of measure 0), and
\[\tau_{\lambda} \tau_\nu = \tau_{\lambda \nu}\text{ a.e.}\]
 (Note that this implies that up to a set of measure 0, $D_{\lambda \nu} = D_\nu$ whenever $s(\lambda) = r(\nu)$).
\item[(d)] The coding maps satisfy $\tau^m \circ \tau^n = \tau^{m+n}$ for any $m, n \in \N^k$.  (Note that this implies that the coding maps pairwise commute.)
\end{itemize}
\end{defn}

\begin{rmk}
\label{rmk:abs-cts-inverse}
We pause to note that  condition (c) of Definition \ref{def-lambda-SBFS-1} above implies that $D_\lambda=D_{s(\lambda)}$ and $R_\lambda\subset R_{r(\lambda)}$ for $\lambda\in\Lambda$. Also, when $\Lambda$ is a finite 1-graph, the definition of a $\Lambda$-semibranching function system is not equivalent to Definition \ref{def-1-brach-system}. In particular, Definition~\ref{def-lambda-SBFS-1}(b) implies that the domain sets $\{D_v:v\in\Lambda^0\}$ must satisfy $\mu(D_v\cap D_w)=\mu(R_v\cap R_w)=0$ for $v\ne w\in\Lambda^0$, but  Definition~\ref{def-1-brach-system} does not require that the domain sets $D_i$ be mutually disjoint $\mu$-a.e. In fact, Definition~\ref{def-lambda-SBFS-1} implies what is called condition (C-K) in Section~2.4 of \cite{bezuglyi-jorgensen}: up to a measure zero set, 
\begin{equation}
\label{eq-partition}
D_v=\cup_{\lambda\in v\Lambda^m} R_\lambda
\end{equation}
for all $v\in\Lambda^0$ and $m\in\N,$ since $R_v=\tau_v(D_v)=id(D_v)=D_v.$ Also notice that in the above decomposition the  intersections $R_\lambda \cap R_{\lambda'},$ $\lambda \not= {\lambda'},$ have measure zero. This condition  is crucial to making sense of the representation of $C^*(\Lambda)$ associated to the $\Lambda$-semibranching function system (see Theorem \ref{thm:separable-repn} below). As established in Theorem~2.22 of \cite{bezuglyi-jorgensen}, in order to obtain a representation of a 1-graph algebra $C^*(\Lambda)$ from a semibranching function system, one must also assume that the semibranching function system satisfies condition (C-K).

Finally, we also observe that $(\tau^n)^{-1} (E) = \bigcup_{\lambda \in \Lambda^n} \tau_\lambda(E)$ for any measurable $E\subseteq X$.  Therefore, 
\[ \mu \circ (\tau^n)^{-1} << \mu\]
in any $\Lambda$-semibranching function system.
\end{rmk}

As established in \cite{FGKP}, any $\Lambda$-semibranching function system gives rise to a representation of $C^*(\Lambda)$ { via \lq prefixing' and \lq chopping off' operators that satisfy the Cuntz-Krieger relations}.  {Intuitively, a $\Lambda$-semibranching function system is a way of encoding the Cuntz-Krieger relations at the measure-space level: the prefixing map $\tau_\lambda$ corresponds to the partial isometry $s_\lambda \in C^*(\Lambda)$.  We give a precise formula for the representation in Theorem \ref{thm:separable-repn} below.} For brevity, we will often refer to  representations arising from $\Lambda$-semibranching function systems as \emph{$\Lambda$-semibranching representations.}  Note that  a $\Lambda$-semibranching representation will be separable whenever $L^2(X, \mu)$ is separable; this will be the case for all but one of the representations we consider in this paper.

\begin{thm}\cite[Theorem~3.5]{FGKP}\label{thm:separable-repn}
Let $\Lambda$ be a finite $k$-graph with no sources and suppose that we have a $\Lambda$-semibranching function system on a certain measure space $(X,\mu)$ with prefixing maps $\{\tau_\lambda\}_{\lambda \in \Lambda}$ and coding maps $\{\tau^m:m\in \N^k\}$. For each $\lambda\in\Lambda$, define an operator $S_\lambda$  on $L^2(X,\mu)$ by
\[
S_\lambda\xi(x)=\chi_{R_\lambda}(x)(\Phi_{\tau_\lambda}(\tau^{d(\lambda)}(x)))^{-1/2} \xi(\tau^{d(\lambda)}(x)).
\]
Then the operators $\{S_\lambda:\lambda\in\Lambda\}$ generate a representation $\pi$ of $C^*(\Lambda)$, and $\pi$ is separable.
\label{thm:SBFS-repn}
\end{thm}

\begin{example}\label{example:SBFS-M}
Here we describe the standard $\Lambda$-semibranching function system on the measure space $(\Lambda^\infty, M)$ for a finite strongly connected $k$-graph $\Lambda$, using the measure $M$ of Equation \eqref{eq:M}. The prefixing maps $\{\sigma_\lambda:Z(s(\lambda))\to Z(\lambda)\}_{\lambda\in\Lambda}$ are given by
\begin{equation}
\label{prefixmaps}
\sigma_\lambda(x)=\lambda x,
\end{equation}
where $\lambda x \in \Lambda^\infty$ is defined by $\lambda x (0, m) = \lambda(0,m)$ if $d(\lambda ) \geq m$, and $\lambda x(0, m) = \lambda x( 0, m - d(\lambda))$ if $m \geq d(\lambda)$,
and the coding maps $\{\sigma^m:\Lambda^\infty\to \Lambda^\infty\}_{m\in \N^k}$ are given as in \eqref{eq:shift-map} of Definition \ref{def:infinite-path}.

Thus, for $\lambda\in\Lambda$, we let $D_\lambda=Z(s(\lambda))$ and $R_\lambda=\sigma_\lambda(D_\lambda)=Z(\lambda)$.  Proposition 3.4 of \cite{FGKP} establishes that $\{\sigma_\lambda:D_\lambda\to R_\lambda\}$ and $\{\sigma^m\}_{m\in \N^k}$ forms a $\Lambda$-semibranching function system on $(\Lambda^\infty,M)$. In particular, one can show that the Radon--Nikodym derivatives of $\sigma_\lambda$ are positive $M$-a.e. on $Z(s(\lambda))$ and they  are given by
\[
\Phi_{\sigma_\lambda}(x)=\rho(\Lambda)^{-d(\lambda)}.
\]

\begin{rmk}
\label{rmk:defn-standard-repn}
As seen in the above Theorem~\ref{thm:separable-repn}, there is a separable representation $\pi =: \pi_S$ of $C^*(\Lambda)$ associated to the standard $\Lambda$-semibranching function system on $(\Lambda^\infty, M)$ of Example \ref{example:SBFS-M}. In this case, $S_\lambda = \pi_S(t_\lambda)$ acts on characteristic functions of cylinder sets by
\begin{equation*}\label{eq:stand-Lambda-repn}
\begin{split}
S_\lambda \chi_{Z(\eta)}(x) &=\chi_{Z(\lambda)}(x)\rho(\Lambda)^{d(\lambda)/2}\chi_{Z(\eta)}(\sigma^{d(\lambda)}(x))\\
&=\rho(\Lambda)^{d(\lambda)/2}\chi_{Z(\lambda\eta)}(x).
\end{split}
\end{equation*}
Then the adjoint $S_\lambda^*$ is given by
\[
S_\lambda^*\chi_{Z(\eta)}(x)=\rho(\Lambda)^{-d(\lambda)/2}\sum_{(\alpha,\beta)\in \Lambda^{\operatorname{min}}(\lambda,\eta)}\chi_{Z(\alpha)}(x).
\]
 We call the separable representation $\pi_S$ associated to this $\Lambda$-semibranching function system on $(\Lambda^\infty, M)$ the \emph{standard $\Lambda$-semibranching representation} of $C^*(\Lambda)$.
\end{rmk}
\end{example}
The following Lemmas are well-known, and will be the technical tool we will use  in many of the Radon--Nikodym  derivative calculations presented in Section \ref{sec:examples_gen_measure}.	
In particular, we will apply these examples to the case where $X = \Lambda^\infty$ and 
$ \mathcal F_n$ is the $\sigma$-algebra generated by the cylinder sets $Z(\lambda)$ with $d(\lambda) = (n, \ldots, n)$.

\begin{lemma}[Kolmogorov Extension Theorem, \cite{kolmogorov, Tum}]
\label{lem-Kolm}
Let $(X, \mathcal{F}_n, \nu_n)_{n\in\N}$ be a sequence of probability measures $(\nu_n)_{n\in \N}$ on the same space $X$, each associated with a $\sigma$-algebra $\mathcal F_n$%{\color{purple} and a measure $\mu_n$ on the same space $X$.} \st{each equipped with a measure $\mu_n$}
; further assume that  $(X, \mathcal{F}_n, \nu_n)_{n\in\N}$  form a projective system, i.e., an inverse limit. Suppose that Kolmogorov's consistency condition holds:
\[ \nu_{n+1}|_{\mathcal F_n} = \nu_n.\]
Then there is a unique extension $\nu$ of the measures $(\nu_n)_{n\in \N}$ to the $\sigma$-algebra $\bigvee_{n\in \N} \mathcal F_n$ generated by $\bigcup_{n\in \N} \mathcal F_n$.
\label{lem:kolmogorov}
\end{lemma}
 In fact, $\nu$ is the unique probability measure which has the given sequence of measures $(\nu_n)_{n\in \N}$ as its marginal distributions with respect to the prescribed filtration $\bigcup_{n\in \N} \mathcal F_n$.

\begin{lemma} (cf.~\cite{bezuglyi-jorgensen}, \cite{SGI} Section 10.2)
\label{lem-RN-der-comp-limit}
	\label{lemma-limit-RN} Let $(X, \mathcal{F}_n, \mu_n)_{n\in \N}$ and $(X, \mathcal F_n, \nu_n)_{n\in \N}$
	be two  sequences of measures on the same space $X$ and same $\sigma$-algebras $(X,{\mathcal F}_n)$.  Suppose that both sequences  form a projective system %, i.e., an inverse limit, 
	and satisfy Kolmogorov's consistency condition, so that by Lemma \ref{lem:kolmogorov}, we have induced measures $\mu, \nu$ on 
the $\sigma$-algebra  $\mathcal{F} := \bigvee_n \mathcal{F}_n$  generated by $\cup_n \mathcal{F}_n$.
  
 Suppose moreover that
	\begin{itemize}
		\item $\nu_n<< \mu_n$ for all $n \in \N$;
		\item The Radon--Nikodym derivative $R_n:= d\nu_n/d\mu_n$ exists and is finite for all $n \in \N$;
		\item $R:= \lim_{n\to \infty} R_n$ exists and is finite.
			\end{itemize}
Then 	$\nu << \mu$ if and only if $R >0$, and $R = d\nu/d\mu$. 
	\end{lemma}

\section{$\Lambda$-semibranching function systems}    
     
\subsection{A new way to construct $\Lambda$-semibranching function systems}
\label{sec:SBFS-revisited}

In order  to construct examples of $\Lambda$-semibranching function systems for a finite $k$-graph $\Lambda$ more readily, we will show that the original definition of $\Lambda$-semibranching function system from \cite{FGKP} can be derived using a procedure that  only involves the $k$-colored edges of $\Lambda$. We present in Theorem~\ref{thm:SBFS-edge-defn} a definition of $\Lambda$-semibranching function systems equivalent to the original definition (Definition \ref{def-lambda-SBFS-1} above). Moreover we give a construction of $\Lambda$-semibranching function system for a product graph $\Lambda$ in Proposition~\ref{prop:product-graph-SBFS}.

The following theorem shows that checking Conditions (a) and (c)  of Definition~\ref{def-lambda-SBFS-1} for arbitrary  $m\in\N^k$ is equivalent to checking the equivalent conditions for the basis elements $e_1,\dots,e_k$ of $\N^k$.

\begin{thm}\label{thm:SBFS-edge-defn}
Let $\Lambda$ be a finite $k$-graph and let $(X,\mathcal{F}, \mu)$ be a measure space. Let $\{e_1,\dots, e_k\}$ be the standard basis of $\N^k$.
For  $1 \leq i \leq k$, suppose we have a semibranching function system $\{ \tau_\lambda: D_\lambda \to R_\lambda \}_{d(\lambda) = e_i}$ on $X$, with associated coding maps 
 $\tau^{e_i}:X\to X$. 
For $\eta \in \Lambda$, write $\eta = \eta_1 \eta_2 \cdots \eta_\ell$ as a sequence of edges, and define
\begin{equation}
\label{eq:tau-lambda-edge-defn}
\tau_\eta := \tau_{\eta_1} \circ \tau_{\eta_2} \circ \cdots \circ \tau_{\eta_\ell}.
\end{equation}
 Then the semibranching function systems $\{\tau_\lambda:d(\lambda) = e_i\}_{i=1}^k$ and coding maps $\{ \tau^{e_i}\}_{i=1}^k$ satisfy Conditions (i) - (v) below if and only if the operators $\{\tau_\eta: \eta \in \Lambda\}$ form a $\Lambda$-semibranching function system, 
   with coding maps $\tau^m:= (\tau^{e_1})^{m_1} \circ (\tau^{e_2})^{m_2} \circ \cdots \circ (\tau^{e_k})^{m_k}$ for $m = (m_1, \ldots, m_k) \in \N^k$.
\begin{itemize}
\item[(i)] 
 For any edges $\lambda, \nu$ with $s(\lambda) = s(\nu) $, we have $D_\lambda = D_\nu$.  Writing $v= s(\lambda) = s(\nu)$, we set 
 \[D_v := D_\lambda = D_\nu,\]
 and we require  $\mu(D_v) > 0$ for all $v \in \Lambda^0$.
\item[(ii)]  
For $v\ne w\in\Lambda^0$, $\mu(D_v\cap D_w)=0$. 
\item[(iii)] Fix $i,j\in \{1,\dots, k\}$. If $\lambda\alpha=\nu\beta$ for $\lambda,\beta\in \Lambda^{e_i}$ and $\nu,\alpha\in \Lambda^{e_j}$, then $R_\alpha\subset D_\lambda$, $R_\beta\subset D_\nu$, and
    \[
    \tau_\lambda\circ \tau_\alpha=\tau_\nu\circ \tau_\beta.
    \]   
\item[(iv)]
For any $1 \leq i, j \leq k$, we have 
$\tau^{e_i} \circ \tau^{e_j} = \tau^{e_j} \circ \tau^{e_i}$.
\item[(v)] For $v\in \Lambda$ and $1\le i\le k$, we have
\[
\mu(D_v\setminus \cup_{g\in v\Lambda^{e_i}} R_{g})=0.
\]
\end{itemize}
\end{thm}

\begin{proof}
First, suppose we are given a $\Lambda$-semibranching function system as in Definition \ref{def-lambda-SBFS-1}.  Condition (c) of Definition \ref{def-lambda-SBFS-1} guarantees Conditions (i) and (iii) in the statement of this Theorem; Condition (ii) follows from Condition (b) and the fact that the maps $\{\tau_v: v \in \Lambda^0\}$ form a semibranching function system.  Condition (d) of Definition \ref{def-lambda-SBFS-1} implies Condition (iv) above. %\st{and Condition (v) follows from Condition (a) of a semibranching function system (see Definition 3.1 of} \cite{FGKP}), \st{applied to the semibranching function system $\{\tau_\lambda: d(\lambda) = e_i\}$.} {\color{red} 
To see (v), fix $i\in \{1,\dots,k\}$ and note that Condition (c) of Definition \ref{def-lambda-SBFS-1}  implies that for $g\in \Lambda^{e_i}$, $R_g\subseteq D_{r(g)}$. Thus, $\cup_{g\in v\Lambda^{e_i}} R_g\subseteq D_v$, and hence $\mu(D_v\setminus \cup_{g\in v\Lambda^{e_i}} R_g)=0$.

For the other direction, suppose that we are given $k$ semibranching function systems $\{\tau_\lambda:  \lambda \in \Lambda, d(\lambda) = e_i\}_{i=1}^k$ with coding maps $\{\tau^{e_i}\}_{i=1}^k$ %{\color{red}  $\{\tau_f:  d(f)  = e_i\}_{i=1}^k$, $\{\tau^{e_i}\}_{i=1}^k$} 
satisfying Conditions (i) - (v) above.  
First fix $\eta\in\Lambda$ and write $\eta=\eta_1\eta_2\dots \eta_{\ell}$ as a sequence of edges. Then Condition (iii) implies that $R_{\eta_j}\subseteq D_{\eta_{j-1}}$ for $2\le j\le \ell$, and hence the formula for $\tau_\eta$ given in \eqref{eq:tau-lambda-edge-defn} is well-defined.\footnote{Note that if $\lambda = \lambda_1 \lambda_2$ with $d(\lambda_1) = \ell, \ d(\lambda_2) = e_j$, then  $R_{\lambda_2} \subseteq D_{s(\lambda_1)}$ by Conditions (iii) and (i), and hence   the composition $\tau_{\lambda_1} \circ \tau_{\lambda_2}$ is well defined.}  
In fact, %if we write $\eta$ as a sequence of edges in a different order, $\eta = \lambda_1 \ldots \lambda_\ell$, then 
Condition (iii) and the factorization property of $k$-graphs imply that $\tau_\eta$ is independent of the decomposition of $\eta$ into edges.
Moreover, recall that since each $\{\tau_\lambda: d(\lambda) = e_i\}$ %{\color{red} $\{\tau_f: d(f) = e_i\}$}
  is a semibranching function system, we have $\tau^{e_i} \circ \tau_\lambda = id_{D_\lambda}$ for all $\lambda \in \Lambda^{e_i}$. %} {\color{red} $\tau^{e_i} \circ \tau_f = id_{D_f}=id_{D_{s(f)}}$ for all $f \in \Lambda^{e_i}$.} \st{
  Consequently, if $\eta \in \Lambda^m$, write $\eta$ as a sequence of edges,  $\eta = \eta_1 \eta_2 \cdots \eta_\ell$ where we list the $m_k$ edges of color $k$ first, then all $m_{k-1}$ edges of color $k-1$, etc.   Also note that $id_{D_\alpha}\circ\tau_\beta$ is well defined for edges $\alpha, \beta$ whenever $s(\alpha)=r(\beta)$, and $id_{D_\alpha}\circ\tau_\beta=\tau_\beta$. Then
\[\tau^m \circ \tau_\eta = (\tau^{e_1})^{m_1} \circ \cdots \circ (\tau^{e_k})^{m_k} \circ \tau_{\eta_1} \circ \cdots \circ \tau_{\eta_\ell} = id_{D_\lambda},\]
since $\tau_k^{m_k} \circ \tau_{\lambda_1} \circ \cdots \circ \tau_{\lambda_{m_k}} = id_{D_{\lambda_{m_k}}}$, and similarly for the other colors.
Hence, $\tau^m$ is a coding map for $\{\tau_\lambda:d(\lambda)=m\}$.

To see that $\{\tau_\lambda: d(\lambda) = m\}$ forms a semibranching function system for each $m \in \N^k$, we proceed by induction.  Note that the case %\st{$|m| = 1$} {\color{red} 
$m=e_i$ for $1\le i\le k$ holds by the hypotheses of the Theorem.  For the case  $m = 0$, we begin by defining 
\[\tau_v = id: D_v \to D_v \ \text{for}\;\;  v \in \Lambda^0.\]
Then $\Phi_v(x)  := \frac{d(\mu \circ \tau_v)}{d\mu} (x)= 1$ for all  $x \in D_v$.  By Condition (ii), in order to check that $\{\tau_v: v \in \Lambda^0\}$ is a semibranching function system, it merely remains to check that $\mu(X \backslash \cup_{v\in \Lambda^0} D_v) = 0$.  By Conditions (ii) and (v), and the fact that $\{\tau_{\lambda}: d(\lambda) = e_i\}$ is a semibranching function system,
\begin{align*}
\mu\left(\bigcup_{v\in \Lambda^0} D_v \right) &= \sum_{v\in \Lambda^0} \mu(D_v) = \sum_{v\in \Lambda^0} \sum_{\lambda \in v\Lambda^{e_i}} \mu(R_\lambda) = \mu(X)
\end{align*}
as desired.

Now, suppose that for every $\ell=(\ell_1,\dots,\ell_k) \in \N^k$ with $|\ell |=\ell_1+\ell_2+\dots +\ell_k \leq n$, we have $\{\tau_\lambda: d(\lambda) = \ell\}$ is a semibranching function system with coding map $\tau^{\ell}$.  Let $m = \ell + e_j$.  Given $\lambda \not= \nu \in \Lambda^m$, write $\lambda = \lambda_1 \lambda_2 , \ \nu = \nu_1 \nu_2$, with $d(\lambda_1) = \ell = d(\nu_1)$ and $d(\nu_2) = d(\lambda_2) = e_j$. 
Then  $\tau_\lambda=\tau_{\lambda_1}\circ \tau_{\lambda_2}$ is well-defined and 
\[R_\lambda := \tau_{\lambda_1}(R_{\lambda_2}) \subseteq R_{\lambda_1}.\]  If $\nu_1 \not= \lambda_1$, then $R_{\lambda} \cap R_\nu \subseteq R_{\lambda_1} \cap R_{\nu_1}$ and hence 
\[\mu(R_\lambda \cap R_\nu) \leq \mu( R_{\lambda_1} \cap R_{\nu_1}) = 0.\]
If $\nu_1 = \lambda_1$, then since $\lambda \not= \nu$ we must have that $\lambda_2 \not= \nu_2$.  Thus, since $\Phi_{\lambda_1} = \frac{d(\mu \circ \tau_{\lambda_1})}{d\mu}$ and {$\mu(R_{\lambda_2}\cap R_{\nu_2})=0$}, we have
% Must do this with integrals since the Radon-Nikodym derivative is not constant!
\begin{align*}
\mu(R_\lambda \cap R_\nu)=   \mu(\tau_{\lambda_1}( R_{\lambda_2} \cap R_{\nu_2})) &  =  \int_{R_{\lambda_2} \cap R_{\nu_2}} 1 \, d(\mu \circ \tau_{\lambda_1}) = \int_{R_{\lambda_2} \cap R_{\nu_2}} 
\Phi_{\lambda_1}\, d\mu = 0.
\end{align*}

To see that $\mu(X \backslash \cup_{\lambda \in \Lambda^m} R_\lambda) = 0$, note that
\[
\bigcup_{\lambda \in \Lambda^m} R_\lambda = \bigcup_{\lambda = \lambda_1 \lambda_2 \in \Lambda^m} \tau_{\lambda_1} (R_{\lambda_2}) = \bigcup_{d(\lambda_1) = \ell} \tau_{\lambda_1}\left( \cup_{\lambda_2 \in s(\lambda_1) \Lambda^{e_j}} R_{\lambda_2} \right) \\
\]
Then Condition (i) and (v) gives 
\[\begin{split}
\bigcup_{d(\lambda_1) = \ell} \tau_{\lambda_1}\left( \cup_{\lambda_2 \in s(\lambda_1) \Lambda^{e_j}} R_{\lambda_2} \right) 
&= \bigcup_{d(\lambda_1) = \ell} \tau_{\lambda_1} (D_{s(\lambda_1)}) \ \text{ almost everywhere} \\
&= \bigcup_{d(\lambda_1) = \ell} R_{\lambda_1} = X \ \text{ almost everywhere.}
\end{split}\]
Thus, $\mu(X \backslash \cup_{\lambda \in \Lambda^m} R_\lambda) = 0$. 

To conclude that $\{\tau_\lambda: d(\lambda) = m\}$ is a semibranching function system,
 we need to show that it satisfies Condition (b) of Definition~\ref{def-1-brach-system}, which states that the Radon--Nikodym derivative $\Phi_\lambda := \Phi_{\tau_{\lambda_1}\circ \tau_{\lambda_2}}$ exists and is positive for all $\lambda = \lambda_1 \lambda_2$ with $d(\lambda_1)=\ell$, $d(\lambda_2)=e_j$. Since $\mu\circ \tau_{\lambda_1} <<\mu$ and $\mu\circ \tau_{\lambda_2}<<\mu$, it is straightforward to see that $\mu\circ \tau_{\lambda_1}\circ \tau_{\lambda_2} << \mu\circ \tau_{\lambda_2}$.
Now we fix a Borel set $E\subset D_{\lambda_2}$, otherwise the following integral is zero, 
and consider 
\[
\int_{X}\chi_E(x)\, d(\mu\circ \tau_{\lambda_1}\circ \tau_{\lambda_2}).
\]
Since $E \subset D_{\lambda_2}$, if $x \in E$ then $\tau_{\lambda_2}(x) =: y \in R_{\lambda_2}$, and so (since $\tau^{e_j} \circ \tau_{\lambda_2} = id_{D_{\lambda_2}}$) we see that we can write every $x \in E$ as  $x = \tau^{e_j}(y)$ for precisely one $y \in R_{\lambda_2}$.  Moreover, the fact that 
$\tau_{\lambda_2} = \tau_{\lambda_2} \circ \tau^{e_j} \circ \tau_{\lambda_2}$ implies that 
 $\tau_{\lambda_2} \circ \tau^{e_j} = id_{R_{\lambda_2}}$, so
\[\begin{split}
\int_{X}\chi_E(x)\, d(\mu\circ \tau_{\lambda_1}\circ \tau_{\lambda_2})(x)
&=\int_X \chi_E(\tau^{e_j}(y))\, d(\mu\circ \tau_{\lambda_1}\circ \tau_{\lambda_2})(\tau^{e_j}(y))\\
&=\int_X(\chi_E\circ \tau^{e_j})(y)\, d(\mu\circ \tau_{\lambda_1})(y).
\end{split}\]
Since $\mu_\circ \tau_{\lambda_1} << \mu$, the above integral becomes
\[
\int_X(\chi_E\circ \tau^{e_j}(y) \Phi_{\tau_{\lambda_1}}(y)\, d\mu(y).
\]
%Since $y\in E'\subseteq R_{\lambda_2}$, we let
Returning to our original notation, write $y=\tau_{\lambda_2}(x)$ for some $x\in E\subset D_{\lambda_2}$; now we have
\[\begin{split}
\int_X(\chi_E\circ \tau^{e_j})(y) \Phi_{\tau_{\lambda_1}}(y)\, d\mu(y)
&=\int_X (\chi_E\circ \tau^{e_j})(\tau_{\lambda_2}(x)) \Phi_{\tau_{\lambda_1}}(\tau_{\lambda_2}(x))\, d\mu (\tau_{\lambda_2}(x))\\
&=\int_X \chi_E(x) (\Phi_{\tau_{\lambda_1}}\circ \tau_{\lambda_2})(x)\, d(\mu\circ \tau_{\lambda_2})(x).
\end{split}\]
So we have
\[
\int_{X}\chi_E(x)\, d(\mu\circ \tau_{\lambda_1}\circ \tau_{\lambda_2})
=\int_X \chi_E(x) (\Phi_{\tau_{\lambda_1}}\circ \tau_{\lambda_2})(x)\, d(\mu\circ \tau_{\lambda_2})(x).
\]
Thus, by uniqueness of Radon--Nikodym derivatives and the fact that $ \mu \circ \tau_{\lambda_1} \circ \tau_{\lambda_2} << \mu \circ \tau_{\lambda_2}$, we have   
\[
\frac{d(\mu\circ \tau_{\lambda_1}\circ\tau_{\lambda_2})}{d(\mu\circ \tau_{\lambda_2})}=\Phi_{\tau_{\lambda_1}}\circ \tau_{\lambda_2}.
\]
Therefore $\Phi_\lambda := \Phi_{\tau_{\lambda_1}\circ \tau_{\lambda_2}}$ exists and
\[
\Phi_\lambda = \Phi_{\tau_{\lambda_1}\circ \tau_{\lambda_2}}=\frac{d(\mu\circ \tau_{\lambda_1}\circ\tau_{\lambda_2})}{d(\mu\circ \tau_{\lambda_2})}\, \frac{d(\mu\circ \tau_{\lambda_2})}{d\mu}=(\Phi_{\tau_{\lambda_1}}\circ \tau_{\lambda_2})(\Phi_{\tau_{\lambda_1}}),
\]
which is positive since $\Phi_{\tau_{\lambda_1}}$ and $\Phi_{\tau_{\lambda_2}}$ are positive. Hence $\{\tau_{\lambda}:d(\lambda)=\ell + e_j\}$ forms a semibranching function system.
Therefore by induction $\{\tau_\lambda: d(\lambda) =m\}$ forms a semibranching function system for all $m \in \N^k$. This completes the proof that Condition (a) holds.   Note that Condition (b) holds by construction and by Condition (i); Condition (c) holds by construction, Condition (v), and the fact that $\tau_\lambda$ is well defined.  Similarly, Condition (d) holds by construction and by Condition (iv), completing the proof of the Theorem.
\end{proof}

\begin{cor}\label{cor:repn_S}
Let $\Lambda$ be a finite $k$-graph with no sources and let $(X,\mathcal{F}, \mu)$ be a measure space. 
For  each $1 \leq i \leq k$, suppose we have a semibranching function system  $\{\tau_f: D_f\to R_f\}_{d(f)=e_i}$ on $(X,\mu)$ with associated coding map  $\tau^{e_i}:X\to X$ satisfying Conditions (i)--(v) in  Theorem  \ref{thm:SBFS-edge-defn}. For each $1\le i \le k$ and $f\in \Lambda^{e_i}$, define $S_f\in B(L^2(X,\mu))$ by
\begin{equation}\label{eq:standard_S_f}
S_f \xi(x)=\chi_{R_f}(x)(\Phi_{\tau_f}(\tau^{e_i}(x))^{-1/2}\xi(\tau^{e_i}(x)).
\end{equation}
Then the collection of operators $\{S_f: d(f)=e_i\}_{i=1}^k$ generate a representation of $C^*(\Lambda)$.
\end{cor}

\begin{proof}
By Theorem \ref{thm:SBFS-edge-defn}, we obtain a $\Lambda$-semibranching function system on $(X,\mu)$, and thus by Theorem 3.5 of \cite{FGKP}, we have an associated representation of $C^*(\Lambda)$ on $L^2(X, \mu)$.  When we evaluate the formula from \cite{FGKP} Theorem 3.5 on paths $f \in \Lambda$  with $|d(f)| = 1$ we obtain the formula for $S_f$ given in the statement of the Corollary.  Moreover, using (CK2), we can compute $S_\lambda$ for any $\lambda \in \Lambda$ once we know the formulas for $\{S_f: f\in \Lambda, |d(f)| = 1\}$.  The fact that the operators $S_\lambda$ arise from the $\Lambda$-semibranching function system induced by $\{ \{\tau_f: D_f\to R_f\}_{d(f)=e_i}\}_{i=1}^k$ guarantees the necessary commutativity properties to ensure that $S_\lambda$ is well defined.  Namely, suppose $\lambda = f_1 f_2 = g_2 g_1$ for $f_i, g_i$ edges of degree $e_i$ in $\Lambda$.  Then Theorem 3.5 of \cite{FGKP} tells us that   $ S_\lambda = S_{f_1} \circ S_{f_2} = S_{g_2} \circ S_{g_1}$, so writing $S_\lambda$ as a composition of operators $S_f$ for an edge $f$ gives the same formula as in \cite{FGKP}, and moreover is independent of the choice of factorization of $\lambda$ into edges.
\end{proof}

Now we describe how to construct a $\Lambda$-semibranching function system when $\Lambda$ is given as a product graph as follows. 

\begin{defn}(See \cite[Proposition~1.8]{KP} and \cite[Proposition~5.1]{Kang-Pask})
Let $(\Lambda_1, d_1)$ and $(\Lambda_2, d_2)$ be $k_1$- and $k_2$-graphs respectively. We  define the \emph{product graph}  $(\Lambda_1\times\Lambda_2, d_1\times d_2)$ to consist of the product category $\Lambda_1 \times \Lambda_2$,  with degree map $d_1\times d_2 : \Lambda_1\times \Lambda_2\to \N^{k_1+k_2}$ given by $d_1\times d_2(\lambda_1,\lambda_2)=(d(\lambda_1),d(\lambda_2))\in \N^{k_1}\times \N^{k_2}$ for $\lambda_1\in\Lambda_1$ and $\lambda_2\in \Lambda_2$.
\end{defn}
According to Proposition~1.8 of \cite{KP}, the product graph $\Lambda_1\times \Lambda_2$ in the above definition is a $(k_1+k_2)$-graph, and  the associated $C^*$-algebra is given by 
\[C^*(\Lambda_1\times\Lambda_2)\cong C^*(\Lambda_1)\otimes C^*(\Lambda_2)\] by Corollary~3.5 of \cite{KP}. 
Also Theorem~5.3 of \cite{Kang-Pask} implies that $\Lambda_1\times\Lambda_2$ is a finite $(k_1+k_2)$-graph with no sources if and only if $\Lambda_i$ is a finite $k_i$-graph with no sources for $i=1,2$.  

Notice that $(\Lambda_1 \times \Lambda_2)^0 = \Lambda_1^0 \times \Lambda_2^0$. Moreover,
paths in $\Lambda_1\times \Lambda_2$ of degree $e_i\in \N^{k_1}$ for the basis vector $e_i\in \N^{k_1}$ can be described as follows.  
We fix $v_1, w_1\in \Lambda_1^0$ and  $v_2, w_2\in \Lambda_2^0$, and we write $e_i$ as $(e_i,0)\in \N^{k_1}\times \N^{k_2}$. Then the paths of degree $e_i$ with  range $(v_1, v_2)$ and  source $(w_1, w_2)$ are given by
\[(v_1, v_2)( \Lambda_1 \times \Lambda_2)^{(e_i, 0)} (w_1, w_2) = \begin{cases} \emptyset, & w_2 \not= v_2 \\ v_1 \Lambda_1^{e_i} w_1 , & w_2 = v_2. \end{cases} \]

Similarly,  if $e_j$ is a basis vector for $\N^{k_2}$, then  
\[(v_1, v_2)( \Lambda_1 \times \Lambda_2)^{(0,e_j)} (w_1, w_2) = \begin{cases} \emptyset, & w_1 \not= v_1 \\ v_2 \Lambda_2^{e_j} w_2 , & w_1= v_1. \end{cases} \]
 Thus, if we choose an ordering of the vertices of $\Lambda_i$ for $i= 1, 2$ and then list the vertices of $\Lambda_1 \times \Lambda_2$ lexicographically, the vertex matrices $A_i$ of $\Lambda_1 \times \Lambda_2$ are given by 
 \[A_i = M_i \otimes I_{k_2}, \ 1 \leq i \leq k_1; \qquad A_{k_1 + j} = I_{k_1} \otimes N_j, \ 1 \leq j \leq k_2,\]
 where $\{M_i\}_{i=1}^{k_1}$ are the vertex matrices for $\Lambda_1$ and $\{N_j\}_{j=1}^{k_2}$ are the vertex matrices for $\Lambda_2$.
 
Since the product graph $\Lambda_1\times \Lambda_2$ is a $(k_1+k_2)$-graph, it satisfies the factorization property and it can be described as follows. Suppose that $\lambda \in (v_1, v_2) (\Lambda_1 \times \Lambda_2)^{(e_j,0)}(w_1, v_2)$ where $e_j$ is a basis vector for $\N^{k_1}$, and $\nu \in (w_1, v_2) (\Lambda_1 \times \Lambda_2)^{(0,e_\ell)}(w_1, w_2)$ where  $e_\ell$ is a basis vector for $\N^{k_2}$.  
Then $\lambda$ and $\nu$ are composable since $s(\lambda)=(w_1,v_2)=r(\nu)$, and
  $\lambda$ corresponds to a morphism $\lambda_1 \in v_1 \Lambda^{e_j}_1 w_1$, and $\nu$ corresponds to a morphism $\nu_2 \in v_2 \Lambda^{e_\ell}_2 w_2$.
 
Then the factorization property of $\Lambda_1\times \Lambda_2$ implies that there exist $\tilde{\nu}\in (v_1,v_2)(\Lambda_1\times \Lambda_2)^{(0,e_\ell)} (v_1,w_2)$ and $\tilde \lambda \in (v_1, w_2)(\Lambda_1 \times \Lambda_2)^{(e_j,0)}(w_1, w_2)$ such that 
 \[ \lambda \nu = \tilde{\nu} \tilde{\lambda}.\]
 Note that $\tilde{\nu}\in (v_1,v_2)(\Lambda_1\times \Lambda_2)^{(0,e_\ell)} (v_1,w_2)$ corresponds to $\nu_2$, and $\tilde \lambda \in (v_1, w_2)(\Lambda_1 \times \Lambda_2)^{(e_j,0)}(w_1, w_2)$ corresponds to $\lambda_1$.
 
 The following proposition describes how to construct a $\Lambda_1\times \Lambda_2$-semibranching function system when we have $\Lambda_1$ and $\Lambda_2$-semibranching function systems on measure spaces $(X_1, \mu_1)$ and $(X_2, \mu_2)$ respectively.  In other words, Proposition \ref{prop:product-graph-SBFS} enables us to construct a wealth of $\Lambda$-semibranching function systems out of a few examples, such as the examples provided in Sections \ref{sec:examples_a} and \ref{sec:examples_gen_measure} below.
 
 \begin{prop}\label{prop:product-graph-SBFS}
For $i=1,2$, let $\Lambda_i $ be a $k_i$-graph with a $\Lambda_i$-semibranching function system $\{\tau^i_{\lambda}: \lambda \in \Lambda_i\}$ on $(X_i, \mu_i)$, with coding maps $\tau^{i, e_j}$ for $1 \leq j \leq k_i$.  For $\lambda \in \Lambda_1$ with $|d(\lambda)|=1$, let $\{\lambda^1_v: v \in \Lambda_2^0\}$ denote the corresponding edges in $\Lambda_1 \times \Lambda_2$, with $s(\lambda_v^1) = (s(\lambda), v)$ and $r(\lambda_v^1) = (r(\lambda), v)$; similarly for $\nu\in \Lambda_2$ but with $s(\nu^2_u)=(u,s(\nu))$, $r(\nu^2_u)=(u,r(\nu))$, where $u\in \Lambda_1^0$.    For $w \in \Lambda_1^0, v \in \Lambda_2^0$, define $D_{w,v} \subseteq X_1 \times X_2$ by 
\[D_{w, v} := D_w \times D_v \subseteq X_1 \times X_2.\]
 Then, define prefixing maps $\tau_{\lambda_v^1}, \tau_{\eta_w^2}$ on $X_1 \times X_2$  by %$\Lambda_1 \times \Lambda_2$
\[\tau_{\lambda^1_v} (x, y) := \chi_{D_v}(y ) \cdot (\tau_\lambda^1(x), y), \quad  \tau_{\eta^2_w}(x,y) := \chi_{D_w}(x) \cdot (x, \tau^2_\eta(y) ),\]
and coding maps $\tau^{e_j}(x, y) = (\tau^{1, e_j}(x), y)$ if $1 \leq j \leq k_1$, or $\tau^{e_j}(x,y) = (x, \tau^{2, e_{j -k_1}}(y))$ if $k_1 < j \leq k_1 + k_2$.
The prefixing maps $\{\tau_{\lambda_v^1}, \tau_{\eta_w^2}: v \in \Lambda_1^0, w \in \Lambda_2^0, |d(\lambda)| = |d(\eta)| = 1\}$ and the coding maps $\{ \tau^{e_j}\}$ satisfy Conditions (i) - (v) of Theorem \ref{thm:SBFS-edge-defn} and thus give rise to a $\Lambda_1 \times \Lambda_2$-semibranching function system.
\end{prop}
\begin{proof}
By construction, $D_{\lambda_u^1} = D_{s(\lambda)} \times D_u = D_{s(\lambda), u} =D_{s(\lambda_u^1)}$ and $D_{\lambda_u^2} = D_u \times D_{s(\lambda)} = D_{u, s(\lambda)}=D_{s(\lambda^2_u)}$; since we began with $\Lambda_i$-semibranching function systems on $X_i$, for $i=1,2$, Conditions (i), (ii), (iv), and (v) of Theorem \ref{thm:SBFS-edge-defn} immediately follow.
To see Condition (iii), observe  that any pair $\lambda \in \Lambda_1, \nu \in \Lambda_2$ gives rise to exactly two composable pairs in $\Lambda_1 \times \Lambda_1$, namely $( \lambda_{r(\nu)}^1,\nu_{s(\lambda)}^2)$ and $(\nu_{r(\lambda)}^2 ,\lambda_{s(\nu)}^1)$ since $s(\lambda^1_{r(\nu)})=(s(\lambda),r(\nu))=r(\nu^2_{s(\lambda)})$ and $s(\nu^2_{r(\lambda)})=(r(\lambda),s(\nu))=r(\lambda^2_{s(\nu)})$.
The factorization rule for product graphs implies that
 \[ \lambda_{r(\nu)}^1\nu_{s(\lambda)}^2 = \nu_{r(\lambda)}^2 \lambda_{s(\nu)}^1 \in \Lambda_1 \times \Lambda_2.\]
 Consequently, $\tau_{\lambda_{r(\nu)}^1} \circ \tau_{\nu_{s(\lambda)}^2} = (\tau_\lambda, \tau_\nu) = \tau_{\nu_{r(\lambda)}^2} \circ \tau_{\lambda_{s(\nu)}^1 }$, so Condition (iii) holds.
 \end{proof}

\subsection{{Examples of $\Lambda$-semibranching function systems on Lebesgue measure spaces}}
\label{sec:examples_a}
{In this section, we describe a few  examples of $\Lambda$-semibranching function systems for finite $2$-graphs $\Lambda$.  In confirming that our examples are indeed $\Lambda$-semibranching function systems, we rely heavily on the characterization given in Theorem \ref{thm:SBFS-edge-defn}.
%{\color{purple}\st{Note that the measure spaces we use below are primarily the unit interval or unit square equipped with Lebesgue measure, and we give an example with non-constant Radon--Nikodym derivatives in Example}}~\ref{ex:non-cst-RN}. 

%Our main focus in the present Section \ref{sec:examples_a} is to hint at the flexibility and diversity offered by the $\Lambda$-semibranching function systems, by showcasing a variety of  examples of $\Lambda$-semibranching function systems on familiar measure spaces for 1- and 2-graphs $\Lambda$.  

\begin{example}
\label{ex:sbfs-3v8e}
Consider the $2$-graph $\Lambda$ given in Example~7.7 of \cite{LLNSW} with the following skeleton.
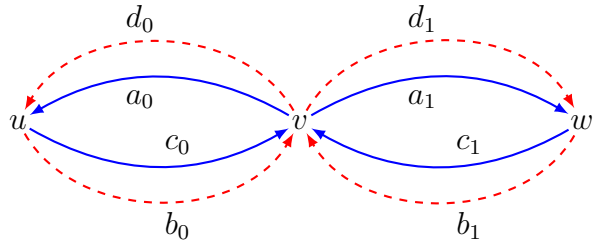
\[
\begin{tikzpicture}[scale=2.5]
\node[inner sep=0.5pt, circle] (u) at (0,0) {$u$};
\node[inner sep=0.5pt, circle] (v) at (1.5,0) {$v$};
\node[inner sep=0.5pt, circle] (w) at (3,0) {$w$};
\draw[-latex, thick, blue] (v) edge [out=150, in=30] (u); %a_0
\draw[-latex, thick, blue] (u) edge [out=-30, in=210] (v); %c_0
\draw[-latex, thick, blue] (v) edge [out=30, in=150] (w); %a_1
\draw[-latex, thick, blue] (w) edge [out=210, in=-30] (v); %c_1
\draw[-latex, thick, red, dashed] (v) edge [out=120, in=60] (u); %d_0
\draw[-latex, thick, red, dashed] (u) edge [out=-60, in=240] (v); %b_0
\draw[-latex, thick, red, dashed] (v) edge [out=60, in=120] (w); %d_0
\draw[-latex, thick, red, dashed] (w) edge [out=240, in=-60] (v); %b_0
\node at (0.65, 0.12) {\color{black} $a_0$}; 
\node at (0.85, -0.12) {\color{black} $c_0$};
\node at (2.15, 0.12) {\color{black} $a_1$};
\node at (2.4, -0.12) {\color{black} $c_1$};
\node at (0.65, 0.55) {\color{black} $d_0$}; 
\node at (0.85, -0.55) {\color{black} $b_0$};
\node at (2.15, 0.55) {\color{black} $d_1$};
\node at (2.4, -0.55) {\color{black} $b_1$};
\end{tikzpicture}
\]
Here the blue and solid edges have degree $e_1$, and the red and dashed edges have degree $e_2$.
The factorization property of $\Lambda$ is given by, for $i=0,1$,
\[
a_ib_i=d_ic_i, \quad a_ib_{1-i}=d_ic_{1-i},\quad\text{and} \quad c_id_i=b_{1-i}a_{1-i}.
\]
In particular, 
\begin{equation}\label{eq:ex3_FP}
\begin{split}
&a_0b_0=d_0c_0, \quad a_1b_1=d_1c_1, \quad a_1b_0=d_1c_0,\\
&a_0b_1=d_0c_1, \quad c_0d_0=b_1a_1,\quad c_1d_1=b_0a_0
\end{split}
\end{equation}
Let $X=(0,1)$ be the unit open interval with  Lebesgue $\sigma$-algebra and measure  $\mu$.

 Let $D_u=(0,\frac{1}{3})$, $D_v=(\frac{1}{3}, \frac{2}{3})$ and $D_w=(\frac{2}{3}, 1)$. Then $\mu(X\setminus (D_u\cup D_v\cup D_w))=0$ and $\mu(D_i\cap D_j)=0$ for $i\ne j$ and $i,j\in \{u,v,w\}$, which gives Condition (i) and (ii) of Theorem~\ref{thm:SBFS-edge-defn}. We first define prefixing maps for blue (solid) edges;
\[\begin{split}
\tau_{a_0}(x)=\frac{3x-1}{3}\quad\;\;&\text{for $x\in D_{a_0}=D_v=\big(\frac{1}{3},\frac{2}{3}\big)$,}\\
\tau_{a_1}(x)=\frac{3x+1}{3}\quad\;\;&\text{for $x\in D_{a_1}=D_v=\big(\frac{1}{3},\frac{2}{3}\big)$,}\\
\tau_{c_0}(x)=\frac{x+1}{2}\quad\;\;&\text{for $x\in D_{c_0}=D_u=\big(0,\frac{1}{3}\big)$,}\\
\tau_{c_1}(x)=\frac{x}{2}\quad\;\;&\text{for $x\in D_{c_1}=D_w=\big(\frac{2}{3},1\big)$.}\\
\end{split}\]
Then the range sets are
\[
R_{a_0}=\big(0,\frac{1}{3}\big), \quad R_{a_1}=\big(\frac{2}{3},1\big),\quad R_{c_0}=\big(\frac{1}{2}, \frac{2}{3}\big), \quad\text{and}\quad \ R_{c_1}=\big(\frac{1}{3}, \frac{1}{2}\big).
\]
Thus, up to sets of measure zero, $D_u=R_{a_0}$, $D_v=R_{c_0}\cup R_{c_1}$, and $D_w=R_{a_1}$. So Condition (v) is satisfied for the degree $e_1$. 
 Moreover, for $e\in \{a_0, a_1, c_0, c_1\}$, the Radon--Nikodym derivative of $\tau_e$ on $D_e$ is given by
\begin{equation*}
\Phi_{{e}}(x) = \inf_{x \in E \subseteq D_{e}} \frac{(\mu\circ \tau_{e})(E)}{\mu(E)}\\
= \inf_{x \in E \subseteq D_{e}}\left\{ \begin{array}{cl}\frac{\frac{1}{2} \mu(E)}{\mu(E)}, & e = c_0, c_1 \\
\frac{\mu(E)}{\mu(E)}, & e = a_0, a_1
\end{array}\right.
= \left\{ \begin{array}{cl} \frac{1}{2}, & e = c_0, c_1 \\
1, & e = a_0, a_1
\end{array}\right.
\end{equation*}
since $\tau_e$ is linear for all $e \in \{a_0, a_1, c_0, c_1\}$.
Now define $\tau^{e_1}$ by
\[
\tau^{e_1}(x)=\begin{cases} \tau_{a_0}^{-1}(x)\quad\text{for}\;\; x\in R_{a_0} \\ \tau_{a_1}^{-1}(x)\quad\text{for}\;\; x\in R_{a_1}\\ \tau_{c_0}^{-1}(x)\quad\text{for}\;\; x\in R_{c_0} \\ \tau_{c_1}^{-1}(x)\quad\text{for}\;\; x\in R_{c_1}
\end{cases}
\]
Then $\tau^{e_1}$ is a coding map for $\{\tau_f: d(f)=e_1\}$. Therefore $\{\tau_f:D_f\to R_f, d(f)=e_1\}$ is a semibranching function system on $(X,\mu)$. Similarly, we define a semibranching function system for red (dashed) edges as follows.
\[\begin{split}
\tau_{d_0}(x)=\frac{-3x+2}{3}\quad\;\;&\text{for $x\in D_{d_0}=D_v=\big(\frac{1}{3},\frac{2}{3}\big)$,}\\
\tau_{d_1}(x)=\frac{-3x+4}{3}\quad\;\;&\text{for $x\in D_{d_1}=D_v=\big(\frac{1}{3},\frac{2}{3}\big)$,}\\
\tau_{b_0}(x)=\frac{-x+1}{2}\quad\;\;&\text{for $x\in D_{b_0}=D_u=\big(0,\frac{1}{3}\big)$,}\\
\tau_{b_1}(x)=\frac{-x+2}{2}\quad\;\;&\text{for $x\in D_{b_1}=D_w=\big(\frac{2}{3},1\big)$.}\\
\end{split}\]
Then 
$R_{d_0}=\big(0,\frac{1}{3}\big),\quad R_{d_1}=\big(\frac{2}{3},1\big),\quad R_{b_0}=\big(\frac{1}{3},\frac{1}{2}\big),\quad \text{and}\quad R_{b_1}=\big(\frac{1}{2},\frac{2}{3}\big)$.
Thus, $D_u=R_{d_0}$, $D_v=R_{b_0}\cup R_{b_1}$ and $D_w=R_{d_1}$, so Condition (v) is satisfied. Also we have $\mu(X\setminus (R_{d_0}\cup R_{d_1}\cup R_{b_0}\cup R_{b_1}))=0$ and $\mu(R_i\cap R_j)=0$ for $i\ne j$ and $i,j\in \{d_0,d_1,b_0,b_1\}$. For $e\in \{d_0,d_1,b_0,b_1\},$ the Radon--Nikodym derivative $\Phi_g$ is given by 
\begin{equation*}
\Phi_{{e}}(x) = \inf_{x \in E \subseteq D_{e}} \frac{(\mu\circ \tau_{e})(E)}{\mu(E)}\\
= \inf_{x \in E \subseteq D_{e}}\left\{ \begin{array}{cl}\frac{\frac{1}{2} \mu(E)}{\mu(E)}, & e = b_0,b_1 \\
\frac{\mu(E)}{\mu(E)}, & e = d_0, d_1
\end{array}\right.
  = \left\{ \begin{array}{cl} \frac{1}{2}, & e = b_0, b_1 \\
1, & e = d_0, d_1.
\end{array}\right.
\end{equation*}

Now we define $\tau^{e_2}$ similarly by
\[
\tau^{e_2}(x)=\begin{cases} \tau_{d_0}^{-1}(x)\quad\text{for}\;\; x\in R_{d_0} \\ \tau_{d_1}^{-1}(x)\quad\text{for}\;\; x\in R_{d_1} \\ \tau_{b_0}^{-1}(x)\quad\text{for}\;\; x\in R_{b_0} \\ \tau_{b_1}^{-1}(x)\quad\text{for}\;\; x\in R_{b_1}
\end{cases}
\]
Then $\tau^{e_2}$ is a coding map for $\{\tau_g:d(g)=e_2\}$. Thus, $\{\tau_g:D_g\to R_g, d(g)=e_2\}$ is a semibranching function system on $(X,\mu)$. 
One verifies in a straightforward fashition that conditions (iii) and (iv) of Theorem~\ref{thm:SBFS-edge-defn} holds for these prefixing maps.
It follows that the above maps give a $\Lambda$-semibranching function system on $(0,1)$  with Lebesgue measure by Theorem~\ref{thm:SBFS-edge-defn}.
\end{example}

\begin{example}
\label{ex:non-cst-RN}
We present here an example of a $\Lambda$-semibranching function system for the 2-graph of one vertex for which the Radon--Nikodym derivatives are not constant.

Consider the following 2-colored graph (cf. Example 4.1 of \cite{FGKPexcursions}).
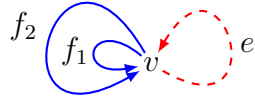
\[
\begin{tikzpicture}[scale=1.7]
 \node[inner sep=0.5pt, circle] (v) at (0,0) {$v$};
\draw[-latex, blue, thick] (v) edge [out=140, in=190, loop, min distance=15, looseness=2.5] (v);
\draw[-latex, blue, thick] (v) edge [out=120, in=210, loop, min distance=40, looseness=2.5] (v);
\draw[-latex, red, thick, dashed] (v) edge [out=-30, in=60, loop, min distance=30, looseness=2.5] (v);
\node at (-0.6, 0.1) {$f_1$}; \node at (-1,0.3) {$f_2$}; \node at (0.75,0.15) {$e$};
\end{tikzpicture}
\]
Then there is a $2$-graph $\Lambda$ with the above skeleton and  factorization rules given by
\begin{equation}\label{eq:ex_facto}
f_1e=ef_2\quad\text{and}\quad ef_1=f_2e.
\end{equation}

Let $X = [0,1]^2$ and let $D_v = (0,1)^2$.  Define 
\begin{equation}\label{eq:pref-maps}
\tau_{f_1}(x, y) = (x, x+y-xy), \quad \tau_{f_2}(x,y) = (x, xy), \quad \tau_e(x,y) = (1-x, 1-y).
\end{equation}
Then $R_{f_1} = \{ (x, y): 0 < x < y\}$ and $R_{f_2} = \{ (x, y): 0 < y < x \}$, and 
\begin{equation}\label{eq:cod-maps}
\tau^{e_2} = \tau_e, \quad \tau^{e_1}(x, y)  = \begin{cases}
(x, y/x) & \text{ if }0 < y < x \\
\left(x, \frac{y-x}{1-x}\right) & \text{ if } 0 < x < y
\end{cases}
\end{equation}
To see that these functions satisfy the conditions of Theorem~\ref{thm:SBFS-edge-defn}, we must check that 
\[ \tau^{e_2} \circ \tau^{e_1} = \tau^{e_1} \circ \tau^{e_2} \quad \text{ and } \quad \tau_{f_i} \circ \tau_{e} = \tau_e \tau_{f_{i+1}}.\]
These equations follow from straightforward calculations.

We now compute the Radon--Nikodym derivatives associated to this $\Lambda$-semibranching function system.  Consider a rectangle $E \subseteq X$ with lower left vertex $(a, b)$ and upper right vertex $(a+\epsilon, b + \delta)$.  Then $\mu(E) = \epsilon \delta$, whereas $ \tau_{f_1}(E) $ is the quadrilateral 
%with vertices 
%\[ (a, a+b -ab); \ (a+\epsilon, a+\epsilon + b - ab -\epsilon a); \ (a+ \epsilon, a + \epsilon + b + \delta - ab - \epsilon b - \delta a - \epsilon\delta); \ (a, a+ b + \delta - ab - a \delta).\]
%In other words, $\tau_{f_1}(E)$ is
bounded by the lines 
\[x = a, \quad  x = a+\epsilon, \quad y = (1-b-\delta) x + b + \delta, \quad y = (1-b)x + b,\]
so a straightforward calculation tells us that 
\[\mu(\tau_{f_1}(E)) = \delta \epsilon ( 1 - a - \epsilon/2),\] and hence 
$\frac{\mu(\tau_{f_1}(E))}{\mu(E)} = 1-a-\epsilon/2.$
Thus, 
\[\Phi_{f_1}(x,y) = \lim_{E \ni ( x, y)} \frac{\mu_{\tau_{f_1}}(E)}{\mu(E)} = 1-x.\]
Similar calculations to the above show that $\tau_{f_2}(E)$ is the quadrilateral bounded by the lines 
\[x=a, \quad x= a+\epsilon, \quad y =( b+\delta) x , \quad y = b x,\]
and hence 
$\frac{\mu({\tau_{f_2}}(E))}{\mu(E)} = a + \epsilon/2$.  Consequently, 
\[ \Phi_{f_2}(x,y) = x.\]

Since $\tau_e$ is linear, $\Phi_e(x,y) = 1$ for all $(x, y) \in D_v$. Hence the prefixing maps given in \eqref{eq:pref-maps} and coding maps given in \eqref{eq:cod-maps} give a $\Lambda$-semibranching function system by Theorem~\ref{thm:SBFS-edge-defn}.

\end{example}

\section{New classes  of $\Lambda$-semibranching function systems associated to probability measures on $\Lambda^\infty$}
\label{sec:examples_gen_measure}

In this section, 
we change our focus to $\Lambda$-semibranching function systems on the infinite path space $\Lambda^\infty$.  We indicate the variety of possible measures on $\Lambda^\infty$ which give rise to $\Lambda$-semibranching function systems, by using Lemmas \ref{lem-Kolm} and \ref{lem-RN-der-comp-limit} to construct many such measures. 

To be precise, 
 we describe a variety of examples of $\Lambda$-semibranching function systems on measure spaces of the form $(\Lambda^\infty, \mathcal{B}_\Lambda, \mu)$, using the standard prefixing and coding maps  $\{\sigma_{\lambda}\}$ and $\{\sigma^n\}$ given in Equations \eqref{prefixmaps} and \eqref{eq:shift-map},
and compare them to the standard $\Lambda$-semibranching function system of Example \ref{example:SBFS-M}.
We begin by describing examples which arise from  Kakutani's product measure construction \cite{Kaku}.  All of the $\Lambda$-semibranching function systems on $(\Lambda^\infty, \mu)$ that we obtain in this way are equivalent to the standard $\Lambda$-semibranching function system, in the sense that the measure $\mu$ is mutually absolutely continuous with respect to the measure $M$ of Equation \eqref{eq:M}.  

Moreover, as Section 3 of \cite{dutkay-jorgensen-monic} shows how to use Markov measures  to construct many inequivalent representations of $\mathcal O_N$,   we also extend these constructions in this section.  To be precise, we identify a family of 2-graphs $\Lambda$ for which the infinite path space $\Lambda^\infty$ either agrees with the infinite path space associated to $\mathcal O_N$, or to a disjoint union of such infinite path spaces.  We then apply the perspective of \cite[Section 3]{dutkay-jorgensen-monic} to construct Markov measures $\{\mu_x: x \in (0,1)\}$, and associated $\Lambda$-semibranching function systems, which yield a family of inequivalent representations of $C^*(\Lambda)$ on $L^2(\Lambda^\infty, \mu_x)$.  If $x \not= 1/2$, the measure $\mu_x$ is mutually singular to the measure $M$ of \eqref{eq:M}.
(For the definition of the Markov measures we are using, see  Definition 3.1 of \cite{dutkay-jorgensen-monic}, and also Definition \ref{def-Markov-measure} of this paper; for a generalized definition of Markov measures, see \cite{bezuglyi-jorgensen}.)

First, we record {in Proposition \ref{thm-lambda-sbfs-on-the-inf-path-space-via-a-measure}}
a straightforward consequence of the definition of a $\Lambda$-semibranching function system given in Definition \ref{def-lambda-SBFS-1}.  Note that Proposition~\ref{thm-lambda-sbfs-on-the-inf-path-space-via-a-measure} simplifies the work of checking when a probability measure on $\Lambda^\infty$ gives rise to a $\Lambda$-semibranching function system.

\begin{prop} 
\label{thm-lambda-sbfs-on-the-inf-path-space-via-a-measure}	
Let $\Lambda$ be a finite, strongly connected $k$-graph.
Suppose that the infinite path space $\Lambda^\infty$ of $\Lambda$  is endowed with a probability measure $p$ satisfying the following properties:
\begin{itemize}
\item[(a)] The standard  prefixing and coding maps $\{\sigma_\lambda\}_{\lambda\in \Lambda}$, $\{\sigma^m\}_{m \in \N^k}$ on $\Lambda^\infty$ given in Equations \eqref{prefixmaps} and \eqref{eq:shift-map} are measurable maps; 
\item[(b)] For all $v\in \Lambda^0$, we have $p(Z(v))>0$.
\item[(c)] Each of the edge prefixing operators $(\sigma_\lambda)_{\lambda\in \Lambda^{e_i}}$ has positive Radon--Nikodym derivative,
\[ \Phi_{\sigma_\lambda}:= \frac{d(p \circ \sigma_\lambda)}{dp} > 0, \text{ p. a.e. on } Z(s(\lambda)).\]
\end{itemize}
Then the maps $\sigma^n, \sigma_\lambda$ endow $(\Lambda^\infty, p)$ with a  $\Lambda$-semibranching function system. 
\end{prop}	

\begin{proof}
The proof is straightforward and completely analogous to the proof of Proposition 3.4 from \cite{FGKP}.  The only argument which differs slightly is to see that all Radon--Nikodym derivatives $\Phi_{\sigma_\lambda}$ are positive for any $\lambda\in\Lambda$, but that is checked in a straightfoward fashion.  
\end{proof}

\subsection{Kakutani-type probability  measures on $\Lambda^\infty$ }
\label{sec:examples_ab-cont-_measure}

We now apply Proposition \ref{thm-lambda-sbfs-on-the-inf-path-space-via-a-measure} to the 2-graph with one vertex in Example \ref{ex:non-cst-RN}.  To be precise, we use a product measure construction inspired by Kakutani in \cite{Kaku}  
to build a Borel measure on the infinite path space $\Lambda^\infty$ which satisfies the hypotheses of Proposition \ref{thm-lambda-sbfs-on-the-inf-path-space-via-a-measure}.  
Recall from Example \ref{ex:non-cst-RN} the $2$-graph $\Lambda$ with one vertex $v$, and two blue edges $f_1$ and $f_2$ and one red edge $e$ satisfying the factorization relations
$$ef_1=f_2e\;\;\;\text{and}\;\;\;ef_2=f_1 e.$$
For any $\xi\in \Lambda^{\infty},$ we can write $\xi$ uniquely as 
$$\xi\equiv\;eg_1eg_2eg_3\cdots eg_n\cdots $$
where $g_i\in \{f_1,f_2\}.$ % i.e. $\xi$ is equivalent to a path that could be written in that way.
We now fix a sequence of positive numbers $\{p_n=\frac{1}{2}+\gamma_n\}_{n=1}^{\infty}$, where $|\gamma_n|<\frac{1}{2}$, such that  $p_n < 1 $ for all $n$, $\lim_{n\to \infty}p_n=\frac{1}{2},$
and $\sum_{n=1}^{\infty}|\gamma_n|<\infty.$  
Set 
$$q_n=1-p_n= \frac{1}{2}-\gamma_n,$$
 and note that $\{q_n\}_{n=1}^{\infty}$ is also a sequence of positive numbers between $0$ and $1$ that tends to $\frac{1}{2}.$
For each $i \in \N$, define
\begin{equation}\label{eq:seq_alpha}
\alpha_i\;=\begin{cases}p_i=\frac{1}{2}+\gamma_i& \text{if} \;g_i=f_1, \\
          q_i=\frac{1}{2}-\gamma_i\; & \text{if}\;g_i=f_2,\;1\leq i\leq n.
          \end{cases}
\end{equation}
Then we define a function $\mu$ on square cylinder sets $Z(e g_1 e g_2 e \cdots g_n)$   by 
\begin{equation}\label{eq:mu_n}
\mu(Z(eg_1eg_2eg_3\cdots eg_n))=\prod_{i=1}^n\alpha_i.
\end{equation}
Also we define an empty product to be 1, so $\mu(Z(v))=1$ for $v\in \Lambda^0$.

\begin{prop}\label{prop:example_exonevtwoe}
Let $\Lambda$ be the 2-graph of Example  \ref{ex:non-cst-RN}. Let $(\alpha_n)_n$ be a sequence given by \eqref{eq:seq_alpha}, and $\mu$ be the function  associated to $(\alpha_n)_n$ as in \eqref{eq:mu_n}.  Then 
\begin{itemize}
\item[(a)] The function $\mu$ extends uniquely to a Borel probability measure on $\Lambda^\infty$, and the standard prefixing and coding maps $(\sigma_\lambda, \sigma^n)$ endow  $(\Lambda^\infty, \mu)$ with a $\Lambda$-semibranching function system. 
\item[(b)] Each such measure $\mu$ is equivalent to the Perron--Frobenius measure $M$ of Equation \eqref{eq:M}.  
\item[(c)] The $\Lambda$-semibranching representation of $C^*(\Lambda)$ on $L^2(\Lambda^\infty, \mu)$ is unitarily equivalent to the standard $\Lambda$-semibranching representation.  In particular, the $\Lambda$-semibranching representations on such measure spaces $L^2(\Lambda^\infty, \mu)$  are all unitarily equivalent.
\end{itemize}
\label{prop:prod-meas-onevtwoe}
\end{prop}

\begin{proof}
To see (a), recall that $\F_n$ is the $\sigma$-algebra generated by $\{Z(\lambda):d(\lambda)=(n,\dots,n)\}$ and $\F_{n}\subseteq \F_{n+1}$. Thus Lemma~\ref{lem:kolmogorov} implies that $\mu$ induces a measure on $\Lambda^\infty$ if, defining $\nu_n:=\mu|_{\F_n}$, we have $\nu_{n+1}|_{\F_n}=\nu_n$.
This is equivalent to saying that $\mu$ is additive on square cylinder sets.
To see that $\mu$ is a probability measure we observe that $\mu(\Lambda^\infty) = \mu(Z(v))$ is the empty product and hence equal to 1 by definition.

Thus, to see that $\mu$ extends to a Borel probability measure on $\Lambda^\infty,$ it only remains to check that $\mu$ is finitely additive on square cylinder sets. If we define $h_i$ to equal $f_1$ when $g_i = f_2$, and vice versa (so that $h_i, g_i \in \{ f_1, f_2\}$ and $h_i \not= g_i$) then we have $Z(e g_1 \cdots e g_n) = Z(e g_1 \cdots eg_ne   g_{n+1}) \sqcup Z(e g_1 \cdots e g_n e h_{n+1})$, a disjoint union of cylinder sets. Therefore, 
\begin{align*} \mu(Z(e g_1 \cdots eg_n e   g_{n+1})) &+ \mu (Z(e g_1 \cdots eg_n e   h_{n+1})) \\
&= \mu(Z(e g_1 \cdots e g_n ))(1/2 + \gamma_{n+1}) + \mu(Z(e  g_1 \cdots e g_n  ))(1/2 - \gamma_{n+1}) \\
&= \mu (Z(e g_1 \cdots e g_n)).
\end{align*}
Arguing inductively, % and using the hypothesis that $\sum_{i=1}^\infty |\gamma_i| < \infty$ to conclude that $\prod_{i=1}^\infty \alpha_i$ is  always finite, one can now 
we conclude that $\mu$ is finitely additive on square cylinder sets, as claimed. 

We now check that $\mu$ satisfies the hypotheses of Proposition~\ref{thm-lambda-sbfs-on-the-inf-path-space-via-a-measure}.  Since $\mu$ is a Borel measure and the maps $(\sigma_\lambda, \sigma^n)$ are continuous, they are measurable; and we observed above that $\mu(Z(v)) = 1$.
%The first two conditions of Theorem \ref{thm-lambda-sbfs-on-the-inf-path-space-via-a-measure} were established in Proposition \ref{prod:prod-meas-onevtwoe} above.  
%and the third condition is a straightforward calculation, using the fact that 
%\[ Z(e) = Z(e f_1) \sqcup Z(e f_2) \text{ and hence } \mu (Z(e)) = (1/2 + \gamma_1) + (1/2 - \gamma_1) = 1.\]
It remains to check that each of the edge prefixing operators, $\sigma_{f_1}, \sigma_{f_2}, \sigma_{e}$, has positive Radon--Nikodym derivatives. To do so, we will use Lemma 	\ref{lemma-limit-RN}. 

Fix an infinite path $\xi \equiv e g_1 e g_2 e g_3 \cdots$.
%We use the following notation.  If $g_i\in \{f_1,f_2\}$ is fixed, define $h_i$ to be the unique element of $\{e_1,e_2\}\backslash \{g_i\}.$  
%% We already gave the above definition.
Define $\ell_i\in \{0,1\}$ so that $\alpha_i = 1/2 + (-1)^{\ell_i} \gamma_i$, 
and let 
$m_i=1-\ell_i.$ For $N \in \N$, we let $\lambda_N=e g_1 \cdots e g_N$.
Then the factorization rule $e f_i = f_{i+1} e$ implies that 
\[\begin{split}
\sigma_{f_1}(Z(\lambda_{N}))&=\{\zeta =(\zeta_i) \in \Lambda^{\infty}:   \zeta_{2j-1}=e\;\;\text{for}\;\; 1\le j\le N,\;\zeta_2 = f_2, \; \zeta_{2i}=h_i,\;2\;\leq i\leq N\}\\
& = Z(e f_2 e h_1 \cdots e h_N).
\end{split}
\]
Since $g_i\ne h_i\in \{f_1,f_2\}$ as described above,
it follows that 
\[
\mu(\sigma_{f_1}(Z(\lambda_{N}))= (\frac{1}{2} -\gamma_1)\prod_{i=2}^{N+1}[\frac{1}{2}+(-1)^{m_i}\gamma_i]
\]

Since we also have 
$$\mu(Z(\lambda_{N}))=\prod_{i=1}^N[\frac{1}{2}+(-1)^{\ell_i}\gamma_i],$$
it follows that (multiplying numerator and denominator by $2^N$)
\begin{equation}
\label{eq-Judy-RN-example}
\frac{\mu(\sigma_{f_1}Z(\lambda_{N}))}{\mu(Z(\lambda_{N}))}\;=\;\left( (\frac{1}{2} - \gamma_1)\prod_{i=2}^{N+1}[1+(-1)^{m_i}2\gamma_i] \right) /\left( \prod_{i=1}^N[1+(-1)^{\ell_i}2\gamma_i]\right).
\end{equation}
We then have 
\[ \Phi_{f_1}(\xi):=\frac{d(\mu \circ \sigma_{f_1})}{d\mu}(\xi) = \lim_{N \to \infty} \frac{\mu(\sigma_{f_1}(Z(\lambda_N))}{\mu(Z(\lambda_N))}.\]
To see that the Radon--Nikodym derivative $\Phi_{f_1}$ is positive, note that
standard results on infinite products imply that, since $|\gamma_i|<1/2$ and $\sum_{i\in \N} |\gamma_i |< \infty$ by hypothesis, 
\[ 
\lim_{n\to \infty}\prod_{i=2}^n \left( 1+ (-1)^{m_i} 2 \gamma_i \right) \quad\text{and}\quad \lim_{n\to \infty}\prod_{i=1}^n (1+(-1)^{\ell_i}2\gamma_i)
\]
are both finite, positive and nonzero for any sequences $(m_i)_i , (\ell_i)_i \subseteq \{ 0,1\}^\N$.  
Indeed, if we let $L$ be the sum of the logarithmic series associated to the denominator $P=\prod_{i=1}^\infty [1+(-1)^{\ell_i}2\gamma_i]$, then one can check that $L=\ln P=\sum_{i=1}^\infty \ln ([1+(-1)^{\ell_i}2\gamma_i])$ has the same absolute convergence behavior as the series 
\[
\sum_{i=1}^\infty |(-1)^{\ell_i}2\gamma_i|,\;\; 
%\st{i.e.}
equivalently, \ \; \sum_{i=1}^{\infty}|\gamma_i|;
\]
this latter series converges by hypothesis. Thus, the series $\sum_{i=1}^\infty \ln ([1+(-1)^{\ell_i}2\gamma_i])$ converges conditionally to a number $L =\sum_{i=1}^\infty \ln ([1+(-1)^{\ell_i}2\gamma_i])\in \R$. But since $L=\ln P$, it cannot be that $P=0$. Therefore, the Radon--Nikodym derivative 
\[
\Phi_{f_1}(\xi)= \frac{d(\mu \circ \sigma_{f_1})}{d\mu}(\xi) = \lim_{N \to \infty} \frac{\mu(\sigma_{f_1}(Z(\lambda_N))}{\mu(Z(\lambda_N))}
\]
converges and is positive as desired. 

Similar calculations, by  using Lemma 	\ref{lemma-limit-RN},  yield the same conclusion for the Radon--Nikodym derivatives associated to $\sigma_e$ and $\sigma_{f_2}$, showing that all the hypotheses of Proposition~\ref{thm-lambda-sbfs-on-the-inf-path-space-via-a-measure}	 are satisfied in this case.  We conclude that $\mu$ makes $\Lambda^\infty$ into a $\Lambda$-semibranching function system with the standard prefixing and coding maps $(\sigma_\lambda, \sigma^n)$, which proves (a).

To see (b), we now use Kakutani's work on product measures to compare the measures $\mu$ constructed in (a) with the Perron--Frobenius measure $M$ on $\Lambda^\infty$ given in \eqref{eq:M}.  Note first that $M$ is a special case of the measure $\mu$ described above, given by taking $\gamma_i = 0$ for all $i$.  

A moment's reflection shows that  $(\Lambda^{\infty},\mu)$ is measure-theoretically isomorphic to 
\[\left(\prod_{i=1}^{\infty}[\{0,1\}]_i, \prod_{n=1}^{\infty}\mu_i\right),\]
where $\prod_{i=1}^{\infty}[\{0,1\}]_i$ is the set of all sequences consisting of $0$ and $1$ only, and
the measure $\mu_i$ on the $i^{th}$ factor space $\{0,1\}$ is given by 
$$\mu_i(\{0\})=\frac{1}{2}+\gamma_i\quad\text{and}\quad \;\mu_i(\{1\})=\frac{1}{2}-\gamma_i\quad\text{for}\;\; i\in \mathbb N.$$
The isomorphism is given by $\prod_{i\in \N} [\{0,1\}]_i \ni (a_i)_{i\in \N} \mapsto e f_{a_1+1} e f_{a_2+1} e \cdots\in \Lambda^\infty$.
It follows from Corollary 1 of Section 10 of \cite{Kaku} that the measure $\mu$ on $\Lambda^{\infty}$ is equivalent (mutually absolutely continuous)  to the Perron--Frobenius measure $M$ whenever the infinite series
$$\sum_{i=1}^{\infty}\left(\sqrt{\frac{1}{2}}-\sqrt{\frac{1}{2}+\gamma_i}\right)^2+\left (\sqrt{\frac{1}{2}}-\sqrt{\frac{1}{2}-\gamma_i}\right)^2,$$  or equivalently, the infinite series 
$$ \sum_{i=1}^{\infty}\left(1-\frac{\sqrt{1+2\gamma_i}}{2}-\frac{\sqrt{1-2\gamma_i}}{2}\right),
$$
converges.  However, this series converges whenever $\sum_{i \in \N} |\gamma_i| < \infty$. But this is our standing hypothesis, and  hence the measure $\mu$ constructed in this fashion is equivalent to $M$.

To see (c), let $g_\mu \in L^2(\Lambda^\infty, \mu )$ be given by 
\[ g_\mu(x) = \sqrt{\frac{d\mu}{dM}(x)},\]
and define $W_\mu : L^2(\Lambda^\infty, \mu ) \to L^2(\Lambda^\infty, M)$ by $W_\mu (f) = g_\mu  f.$
Then one checks that $W_\mu^*(f) = \frac{f}{g_\mu}$ is given by multiplication by $\sqrt{\frac{dM}{d\mu}(x)}$.

For $\lambda \in \Lambda$, write $S_\lambda^\mu$ for the operator on $L^2(\Lambda^\infty, \mu)$ associated to $\lambda$ via the $\Lambda$-semibranching function system on $(\Lambda^\infty, \mu)$, as in Theorem 3.5 of \cite{FGKP}; that is,  if $d(\lambda) = n$,
\begin{equation}\label{eq:SBFS-rep-formula}
 S_\lambda^\mu(\chi_{Z(\eta)})(x) = \left(\frac{d\mu}{d(\mu \circ \sigma^n)}(x) \right)^{-1/2} \chi_{Z(\lambda \eta)}(x).\end{equation}
Moreover, the formula of $W_\mu$ implies that
\[ W_\mu^* S_\lambda^M W_\mu (\chi_{Z(\eta)} )(x) = S_\lambda^\mu (\chi_{Z(\eta)})(x).
\]
Thus, $L^2(\Lambda^\infty, \mu)$ and $L^2(\Lambda^\infty, M)$ are unitarily equivalent, via the unitary $W_\mu$ which intertwines the two $\Lambda$-semibranching representations, $S_\lambda^\mu$ and $S_\lambda^M$.
It follows that
any   $\Lambda$-semibranching function system on $\Lambda^\infty$ associated to a measure $\mu$ as described above  give rise to a representation of $C^*(\Lambda)$ which is equivalent to the standard $\Lambda$-semibranching representation on $L^2(\Lambda^\infty, M)$.
\end{proof} 

The equivalence of the $\Lambda$-semibranching representations discussed above is an instance of a more general phenemenon. In fact, we can apply the above construction to the 2-graph  in Example~\ref{ex:sbfs-3v8e}, namely $\Lambda_2$, and the 2-graph $\Lambda_{2N}$ described below which is a generalization of $\Lambda_2$.
The key idea is to realize the infinite path space of given 2-graphs as the disjoint union of the infinite product spaces and define a product measure accordingly on each of them.
Since one can check that any such product measure is equvalent to the Perron-Frobenius measure $M$, we only  give the construction of such product measures on $\Lambda_{2N}$.

For each $N \in \N$, the $2$-graph $\Lambda_{2N}$ has $2N +1$ vertices labeled $v, u_1, \ldots, u_N, w_1, \ldots w_N$ with red and blue edges connecting $v$ with each of the vertices $u_i, w_i$, in both directions:
 The 2-colored graph (or skeleton) of $\Lambda_{2N}$ is given as below.
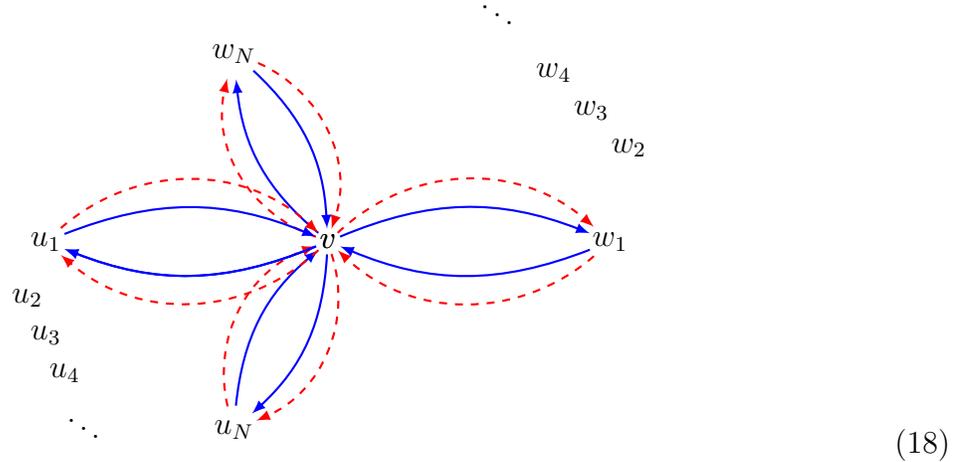
\begin{equation}
\begin{tikzpicture}[scale=2.5]
\node[inner sep=0.5pt, circle] (u_1) at (0,0) {$u_1$};
\node[inner sep=0.5pt, circle] (u_a) at (-.1,-.3) {$u_2$};
\node[inner sep=0.5pt, circle] (u_d) at (0,-.5) {$u_3$};
\node[inner sep=0.5pt, circle] (u_b) at (.1,-.7) {$u_4$};
\node[inner sep=0.5pt, circle] (u_c) at (.2,-.95) {$\ddots$};
\node[inner sep=0.5pt, circle] (v) at (1.5,0) {$v$};
\node[inner sep=0.5pt, circle] (w_1) at (3,0) {$w_{1}$};
\node[inner sep=0.5pt, circle] (w_a) at (3.1,.5) {$w_{2}$};
\node[inner sep=0.5pt, circle] (w_b) at (2.9,.7) {$w_{3}$};
\node[inner sep=0.5pt, circle] (w_d) at (2.7,.9) {$w_{4}$};
\node[inner sep=0.5pt, circle] (w_c) at (2.4,1.25) {$\ddots$};
\node[inner sep=0.5pt, circle] (v) at (1.5,0) {$v$};
\node[inner sep=0.5pt, circle] (w_2) at (1,1) {$w_{N}$};
\node[inner sep=0.5pt, circle] (u_2) at (1,-1) {$u_N$};
\draw[-latex, thick, blue] (v) to[bend left=22] (u_2); 
\draw[-latex, thick, blue] (u_2) to[bend left=22] (v); 
\draw[-latex, thick, blue] (v) to[bend left=22](w_2); %
\draw[-latex, thick, blue] (w_2) to[bend left=22] (v); %
\draw[-latex, thick, red, dashed] (v) to[bend left=42] (u_2); %
\draw[-latex, thick, red, dashed] (u_2) to[bend left=42] (v); %
\draw[-latex, thick, red, dashed] (v) to[bend left=42] (w_2); %
\draw[-latex, thick, blue] (v) to[bend left=22] (u_1); %a_0
\draw[-latex, thick, red, dashed] (w_2) to[bend left=42](v); %
\draw[-latex, thick, blue] (v) to[bend left=22] (u_1); %a_0
\draw[-latex, thick, blue] (u_1) to[bend left=22] (v); %c_0
\draw[-latex, thick, blue] (v) to[bend left=22] (w_1); %a_1
\draw[-latex, thick, blue] (w_1) to[bend left=22](v); %c_1
\draw[-latex, thick, red, dashed] (v) to[bend left=42] (u_1); %d_0
\draw[-latex, thick, red, dashed] (u_1)to[bend left=42] (v); %b_0
\draw[-latex, thick, red, dashed] (v) to[bend left=42] (w_1); %d_0
\draw[-latex, thick, red, dashed] (w_1) to[bend left=42](v); %b_0
%\node at (0.65, 0.12) {\color{black} $a_0$}; 
%\node at (0.85, -0.12) {\color{black} $c_0$};
%\node at (2.15, 0.12) {\color{black} $a_1$};
%\node at (2.4, -0.12) {\color{black} $c_1$};
%\node at (0.65, 0.55) {\color{black} $d_0$}; 
%\node at (0.85, -0.55) {\color{black} $b_0$};
%\node at (2.15, 0.55) {\color{black} $d_1$};
%\node at (2.4, -0.55) {\color{black} $b_1$};
%\node at (-1, 1.12) {\color{black} $\Lambda_{2N}$};
\end{tikzpicture}
\label{eq:Lambda2N-skeleton}
\end{equation}
 There are multiple choices of factorization rules that will make the above skeleton into a 2-graph.  Regardless of the factorization rule we choose, every (finite or infinite) path will have a unique representative as an alternating string of blue (solid) and red (dashed) edges, with the first edge being red. In fact, such a path is completely determined by the sequence of vertices it passes through: we fix a relabeling the vertices $u_i, w_i$ of $\Lambda_{2N}$ by $\{ Q_i\}_{i=1}^{2N}$, and then every infinite path $\xi$ with range $v$ is specified uniquely by a string of vertices  
\begin{equation*} \xi \equiv (v, Q_1, v, Q_3, \ldots ) \text{ where } Q_{2i+1} = u_j \text{ or } w_j  \text{ for some } 1 \leq j \leq N.\label{eq:inf-path-Lambda2N}
\end{equation*}
Similarly, if $r(\xi) \in \{u_j, w_j\,:\,  1\le j\le N \}$, then $\xi \equiv ( Q_0, v, Q_2, v, \ldots) $ for a unique sequence $(Q_{2i})_i \in \{ u_j, w_j: 1 \leq j \leq N\}$ for $i\in \N$.

With this notation, the isomorphism between $\prod_{i\in \N} \Z_{2N} \sqcup \prod_{i \in \N} \Z_{2N}$ and $\Lambda^\infty_{2N}$  is given by mapping a sequence $( a_i)_{i \in \N}$ in the first copy of $\prod_{i\in \N} \Z_{2N}$ to $\xi \equiv (v, Q_{a_1}, v, Q_{a_2}, \ldots)$ and mapping a sequence $(b_i)_{i\in \N}$ in the second copy of $ \prod_{i\in \N} \Z_{2N}$ to the infinite path $\xi \equiv (Q_{b_1}, v, Q_{b_2}, v, \ldots)$.

Thus, given $N$ sequences $\{ (\delta_i^j)_{i\in \N}\}_{j=1}^N$ with $\sum_i |\delta^j_i| < \infty$ for all $j$, we can define an associated
 product
measure $\mu_{2N}$ on $\Lambda^\infty_{2N}$.  Given  $\eta \in \Lambda_{2N}$ with $d(\eta) = (n,n)$, we identify $\eta$ with the string of vertices it passes through:
\begin{equation}\label{eq:2N-paths}
r(\eta) = v \Rightarrow \eta \equiv (v, Q_1, \ldots, Q_{2n-1}, v) \qquad r(\eta) \not= v \Rightarrow \eta \equiv (Q_0, v, Q_2, \ldots, v, Q_{2n}),
\end{equation}
where $Q_i \not= v$.  Then, we define
\begin{equation}\label{eq:measure_2N}
\alpha_i = \begin{cases} \delta^j_i, & Q_i = u_j \\ -\delta^j_i, & Q_i = w_j \end{cases}  \qquad \text{and } \qquad \mu_{2N}(Z(\eta)) =\begin{cases} \prod_{i=1}^n \frac{1 + \alpha_{2i-1}}{2N}, & r(\eta) = v \\   \prod_{i=0}^n \frac{1 + \alpha_{2i}}{2N}, & r(\eta) \not= v. 
\end{cases}\end{equation}

The proof of the following Proposition can be carried out in a similar fashion to the proof of Proposition~\ref{prop:example_exonevtwoe}.

\begin{prop}
\label{prop:inf-path-Lambda-2N}
The formula for $\mu_{2N}$ given in \eqref{eq:measure_2N} defines a measure on $\Lambda^\infty_{2N}$ which is equivalent  to the Perron--Frobenius measure $M$ on $\Lambda^\infty$ of Equation \eqref{eq:M}. The standard prefixing and coding maps $(\sigma^n, \sigma_\lambda)$ make $(\Lambda^\infty_{2N}, \mu_{2N})$ into a $\Lambda$-semibranching function system.  Moreover,  the resulting $\Lambda$-semibranching representation of $C^*(\Lambda)$ on $L^2(\Lambda^\infty, \mu_{2N})$ is equivalent to that associated to the Perron--Frobenius measure $M$ on $\Lambda^\infty$ of Equation \eqref{eq:M}.

\end{prop}

\subsection{Examples of probability measures on $\Lambda^\infty$ that are mutually singular with  the Perron--Frobenius measure}
\label{sec-Markov-measure-Lambda-semibran-0}

%{
%In Section 3 of \cite{dutkay-jorgensen-monic}, the authors outline a procedure for constructing Markov measures  on the infinite path space $\mathcal{K}_N$ of the Cuntz algebra $\mathcal{O}_{N}$, such that the resulting measures  are mutually singular. By results from \cite{FGJKPmonic}, this will imply that that the associated representations are not unitarily equivalent. 

In this section, we will first recall the definition of Markov measure on the infinite path space of Cuntz algebras from \cite{dutkay-jorgensen-monic}, and then we
will apply this  first to the 2-graph $\Lambda$ of Example \ref{ex:non-cst-RN}, and then to the 2-graphs $\Lambda_{2N}$ which is a generalization of the 2-graph given in Example~\ref{ex:sbfs-3v8e}. Indeed   the infinite path spaces of these $2$-graphs are either homeomorphic to $\Lambda^\infty_{\mathcal{O}_N}$ or to a disjoint union of copies of $\Lambda^\infty_{\mathcal O_N}$, which makes our constructions possible.

\begin{defn}[Definition 3.1 of  \cite{dutkay-jorgensen-monic}]
		\label{def-Markov-measure} 
A {\em Markov measure} on the infinite path space  $\Lambda^\infty_{\mathcal{O}_N}$
\[
\Lambda^\infty_{\mathcal{O}_N}=\prod_{i=1}^\infty \Z_N=\{(i_1 i_2\dots)\,:\, i_n\in \Z_N,\;\; n=1,2,\dots\}.
\] of the Cuntz algebra  ${\mathcal{O}_N}$ is defined by a vector 
	$\lambda = (\lambda_0, \ldots, \lambda_{N-1}) $ and an $N \times N$ matrix $T$ such that $\lambda_i > 0$, $T_{i,j} > 0$ for all $i,j \in \Z_N,$ and if 
	$e = (1,1, \ldots ,1)^t$ then $\lambda T=\lambda $ and $Te=e$.
The Carath\'eodory/Kolmogorov extension theorem then implies that  there exists a unique   Borel measure $\mu$ on $\Lambda^\infty_{\mathcal{O}_N}$  extending the measure $\mu_{\mathcal{C}}$ defined on cylinder sets by
\begin{equation}
\mu_{\mathcal{C}} (Z(I)) : = \lambda_{i_1} T_{i_1,i_2} \cdots T_{i_{n-1},i_n},\text{ if }I = i_1 \ldots i_n.
\label{eq-def-markov-measu-Cuntz}
\end{equation}
The extension  $\mu$ is called a \emph{Markov measure}\footnote{For Markov measures in a more general context, see \cite{bezuglyi-jorgensen}.} on $\Lambda^\infty_{\mathcal{O}_N}$.
  \end{defn}

For $N=2,$  
fixed a number $ x \in (0,1)$, we can take  $T=T_x= \begin{pmatrix}
 x & (1-x) \\ (1-x)  & x 
\end{pmatrix}$, and  $\lambda=(1,1)$. The resulting measure will in this case be called $\mu_x$. 
Moreover, if $x \not= x'$, Theorem 3.9 of \cite{dutkay-jorgensen-monic} guarantees that $\mu_x, \mu_{x'}$ are mutually singular.

%We now describe a specific example of a Markov measure on the infinite path space $\mathcal{K}_2$ of $\mathcal{O}_2$ which we use extensively in what follows.
%
%Fix a number $ x \in (0,1)$, the unit interval, and define $T_x=(T_{i,j}) = \begin{pmatrix}
% x & (1-x) \\ (1-x)  & x 
%\end{pmatrix}$.  
%Let $\lambda=(1,1)$ be a row vector with 1 in all entries.
%Then it is straightforward to check that the pair $(T_x, \lambda)$ satisfy 
%\[
%\lambda \,T_x=\lambda ,\quad T_x \,e=e.
%\]
%Then as in Definition~\ref{def-Markov-measure}, the Markov measure $\mu_x$ on $\mathcal{K}_2$ is given  on the cylinder sets by
% \begin{equation}\label{eq:mu_x}
%  \mu_x(Z(i_1 i_2 \cdots i_n)) =  T_{i_1, i_2} T_{i_2, i_3} \cdots T_{i_{n-1}, i_n},
%  \end{equation}
% where $i_j\in \Z_2=\{0,1\}$.  
% Using the fact that the infinite path space of the 2-graph $\Lambda$ in Example \ref{exonevtwoe} is naturally homeomorphic to that of $\mathcal{O}_2$ (by mapping $e f_j$ to $i_j$ for $j=1,2$) we can convert any Markov measure on $\Lambda^\infty_{\mathcal{O}_N}$,  $\mu_x$, into a measure on $\Lambda^\infty$, which we will continue to denote by $\mu_x$.
%} 
% 
%Under this correspondence, the measure $\mu_{1/2}$ satisfies 
%\[ \mu_{1/2}(E) = 2 M(E)\quad \text{for all Borel sets}\;\; E \subseteq \Lambda^\infty. \]

We will now  define Markov measures on the 2-graph  of Example \ref{ex:non-cst-RN},

\begin{prop} 
\label{prop-RN-of-Markov-2-example}
Let $\Lambda$ be the 2-graph given in Example \ref{ex:non-cst-RN}.
Fix a number $x \in (0,1)$, and let $\mu_x$ be Markov measure given by the $2\times 2$ matrix $T_x$ and the vector $\lambda=(1,1)$ as above.  As operators on $L^2(\Lambda^\infty, \mu_{x})$, the prefixing operators  $\sigma_e,\sigma_{f_1},\sigma_{f_2}$ have positive Radon--Nikodym derivatives at any point $z \in \Lambda^\infty$.
 Consequently, the standard prefixing and coding maps make $(\Lambda^\infty, \mu_x)$ into a $\Lambda$-semibranching function system.
\end{prop}

\begin{proof} 
Recall from the proof of Proposition \ref{prop:prod-meas-onevtwoe} the homeomorphism between $\Lambda^\infty$ and $\prod_{i=1}^\infty [\{0,1\}]_i$, given by $\prod_{i\in \N} [\{0,1\}]_i \ni (a_i)_{i\in \N} \mapsto e f_{a_1+1} e f_{a_2+1} e \cdots\in \Lambda^\infty$.  Thus, the Markov measure $\mu_x$ can be viewed as a measure on $\Lambda^\infty$ which satisfies 
\[ \mu_x(Z(e f_i e f_j)) = T_{i, j} = \begin{cases}
x, & i = j \\ 1-x, & i \not= j.
\end{cases} \]
Since $\mu_x$ is a Borel measure and the standard coding and prefixing maps $\sigma^n, \sigma_\lambda$ are local homeomorphisms, they are $\mu_x$-measurable for any $x \in (0,1)$.  

Thus, once we show that the  prefixing operators $\sigma_e,\sigma_{f_1},\sigma_{f_2}$ have positive Radon--Nikodym derivatives (for which we use Lemma~\ref{lemma-limit-RN}), Proposition~\ref{thm-lambda-sbfs-on-the-inf-path-space-via-a-measure} tells us that the standard prefixing and coding maps constitute a $\Lambda$-semibranching function system on $(\Lambda^\infty, \mu_x)$.

Thus, fix $z = e f_{i_1} ef_{i_2}ef_{i_3} \ldots \in \Lambda^\infty$,
and a sequence $(z_n)_n$ of finite paths 
\[
z_n:=e f_{i_1}ef_{i_2}ef_{i_3}\ldots f_{i_n}
\]
such that $z = \bigcap_{n\in \N} Z(z_n)$.  By Lemma \ref{lemma-limit-RN}, for any finite path $g \in \Lambda$, we have 

\[
\frac{d(\mu_x \circ \sigma_g)}{d\mu_x}(z) = \lim_{n\to \infty} \frac{ \mu_{x}(   Z( gz_n  ) ) } {\mu_{x} (Z(z_n ))}.
\]
If we take $g=e$, the factorization rules $e f_i = f_{i-1} e$, for $i= 1,2$, imply that
\[
\lim_n \frac{ \mu_{x}(   Z( gz_n)   ) } {\mu_{x} (Z(z_n) )}=\lim_n \frac{ \mu_{x}(  Z(e f_{i_1-1}e f_{i_2-1}\ldots f_{i_n-1} e) ) } {\mu_{x}( Z(  e f_{i_1}e f_{i_2} \ldots f_{i_n})  ) }=\lim_n \frac{T_{i_1-1,i_2-1}\cdots T_{i_{n-1}-1,i_n-1}} {T_{i_1,i_2} \cdots T_{i_{n-1},i_n}} ,
\]
since $Z(e f_{i_1-1}e f_{i_2-1}\ldots f_{i_n-1} e ) = Z(e f_{i_1-1}e f_{i_2-1}\ldots f_{i_n-1} e f_1  ) \sqcup Z(e f_{i_1-1}e f_{i_2-1}\ldots f_{i_n-1} e f_2)$ has 
\[\mu_x (Z(e f_{i_1-1}e f_{i_2-1}\ldots e f_{i_n-1} e) ) = \mu_x Z(e f_{i_1-1}e f_{i_2-1}\ldots e f_{i_n-1} ) ).\]

Now, observe that for any $i,j \in \Z/2\Z$ we have $T_{i,j} = T_{i-1, j-1}$.  It follows that 
\[ \frac{d(\mu_x \circ \sigma_e)}{d\mu_x}(z) = \lim_{n \to \infty} \frac{T_{i_1-1,i_2-1}\cdots T_{i_{n-1}-1,i_n-1}} {T_{i_1,i_2}\cdots T_{i_{n-1},i_n}} = 1.\]

Similarly, for $j= 1, 2,$  by Lemma 	\ref{lemma-limit-RN},  the Radon--Nikodym derivative
\[
\frac{d(\mu_x \circ \sigma_{f_j})}{d\mu_x}(z) = \lim_{n\to \infty} \frac{ \mu_{x}(   Z( f_j z_n  ) ) } {\mu_{x} (Z(z_n ))} = \lim_{n\to \infty} \frac{T_{j+1, i_1+1} T_{i_1 +1, i_2 + 1} \cdots T_{i_{n-1} + 1, i_n +1}}{T_{i_1, i_2} \cdots T_{i_{n-1}, i_n}} = T_{j+1, i_1 +1}
\]
is positive (indeed, constant on each cylinder set $Z(e f_i)$ for $i = 1, 2$).
\end{proof}

Now recall the 2-graph $\Lambda_{2N}$ from the previous section with the skeleton given in \eqref{eq:Lambda2N-skeleton}.
As described before, $\Lambda^\infty_{2N}$ is isomorphic to $\prod_{i\in \N} \Z_{2N} \sqcup \prod_{i \in \N} \Z_{2N}$.
Also observe that a choice of factorization  on $\Lambda_{2N}$ is equivalent to choosing a permutation $\phi$ of $\{1, \ldots, 2N\}$ such that the red-blue path $(v, Q_i, v)$ equals the blue-red path $(v, Q_{\phi(i)}, v)$.
Having specified such a permutation $\phi$, suppose $\phi$ consists of $d$ cycles; write $c_j$ for the smallest entry in the $j$th cycle.

 Fix $d$ vectors $\{x^j\in \R^{2N}: 0<x_i^j<1\;\;\text{for}\;\; 1\le i\le 2N\}_{j=1}^d$ such that $\sum_{i=1}^{2N} x^j_i = 1$ for each $j$, and define $T_x$ to be the  $2 N \times 2N$ matrix with entries from $(0,1)$ such that 
 \[ T_x(i, j) = x^m_{\phi^{n-1}(j)} \ \text{ if } i = \phi^{n-1}(c_m).\]
By construction, we have $T_x(i,j) = T_x(\phi(i), \phi(j))$ for all $1 \leq i, j \leq 2N$.  Moreover,  the fact that all rows of $T$ sum to 1 implies that $(T, ( 1, 1, \ldots, 1)^T)$ satisfies the conditions given in Definition 3.1 of \cite{dutkay-jorgensen-monic}. Therefore, we have a Markov measure $\mu_x$  associated to $T$ as follows.

\begin{prop}\label{prop:markov-2N}
Let $\Lambda_{2N}$ be a 2-graph with skeleton  \eqref{eq:Lambda2N-skeleton} and factorization rule determined by the permutation $\phi \in S_{2N}$.  For each matrix $T_x$ as above, write $\mu_x$ for the associated measure  on $\Lambda_{2N}^\infty \cong \prod_{i\in \N} \Z_{2N} \sqcup \prod_{i\in \N} \Z_{2N}$, given on a cylinder set in either copy of $\prod_{i\in \N} \Z_{2N} $ by 
\[ \mu_x(Z(a_1 \cdots a_n)) = \prod_{i=1}^{n-1} T_x(a_i, a_{i+1}). \]
Then the standard prefixing and coding maps make $(\Lambda_{2N}^\infty, \mu_x)$ into a $\Lambda$-semibranching function system.  If the vectors $x^m$ are not all constant, $\mu_x$ is mutually singular with respect to the measure $M$ of Equation \eqref{eq:M}.
\end{prop}

\begin{proof}
As above, we merely need to check the Radon--Nikodym derivatives by using Lemma 	\ref{lemma-limit-RN}.  Fix a red edge $e$ with range $Q_i$, and fix a point $\xi \equiv (v, Q_{b_1}, v, Q_{b_2}, \ldots) \in \Lambda^\infty$ (with a  red edge listed first).  Then, 
\begin{align*}
\frac{d(\mu_x \circ \sigma_e)}{d\mu_x}(\xi) &= \lim_{n \to \infty} \frac{\mu_x \circ \sigma_e(Z(v, Q_{b_1}, \ldots, Q_{b_n}, v))}{\mu_x(Z(v, Q_{b_1}, \ldots, Q_{b_n}, v))}\\
&= \lim_{n\to \infty} \frac{\mu_x(Z(Q_i, v, Q_{\phi(b_1)}, v, \ldots, Q_{\phi(b_n)}, v))}{\mu_x(Z(v, Q_{b_1}, \ldots, Q_{b_n}, v))}\\
&= \lim_{n \to \infty}\frac{ T_x(i, \phi(b_1)) \prod_{i=1}^{n-1} T_x(\phi(b_i), \phi(b_{i+1}))}{\prod_{i=1}^{n-1} T_x(b_i, b_{i+1})} \\
&= T_x(i, \phi(b_1)).
\end{align*}
Similarly, if we choose a blue edge $f$  with range $Q_i$,  we calculate:
\[
\frac{d(\mu_x \circ \sigma_f)}{d\mu_x}(\xi)= T_x(i, b_1).
\]
On the other hand, if $\zeta \equiv (Q_{a_1}, v, Q_{a_2}, v, \ldots)$ is an infinite path and $g$ is a blue edge with source $Q_{a_1}$ and range $v$, prefixing $\zeta$ by $g$ and rewriting the result as a sequence of red-blue edges gives $g \zeta \equiv (v, Q_{\phi(a_1)}, v, Q_{\phi(a_2)}, \ldots)$.  It follows by a calculation that 
\[
\frac{d(\mu_x \circ \sigma_g)}{d\mu_x}(\zeta) = \lim_{n \to \infty} \frac{\mu_x \circ \sigma_g(Z(Q_{a_1}, \ldots, Q_{a_n}))}{\mu_x(Z(Q_{a_1}, \ldots, Q_{a_n})))}= 1.
\]
Similarly, for any red edge $h$ with source $Q_{a_1}$, the fact that rewriting $\zeta$ as a blue-red path doesn't change the sequence of vertices it passes through means that $\frac{d(\mu_x \circ \sigma_h)}{d\mu_x}(\zeta) = 1$.

Since all the Radon--Nikodym derivatives are positive, we obtain a $\Lambda$-semibranching function system as claimed.

For the final assertion, one simply observes that since the formula for $M(Z(\lambda))$ only depends on the degree (length) of $\lambda$, it is a rescaling of the measure $\mu_x$ corresponding to the choice $x^m = (\frac{1}{2N}, \ldots, \frac{1}{2N})$ for all $m$.  Thus, if any of the vectors $x^m$ are not constant, Theorem 3.9 of \cite{dutkay-jorgensen-monic} implies that $\mu_x$ is mutually singular with respect to $M$.
\end{proof}

\begin{rmk}
The measure $\mu_x$ used in Proposition \ref{prop:markov-2N} above could equally well be defined for any  2-graph $\Lambda_{2N+1}$ with one central vertex and $2N+1$ peripheral vertices, each connected to the center vertex as in \eqref{eq:Lambda2N-skeleton}.  This is because any such 2-graph (equivalently, any factorization rule for this skeleton) is determined by a permutation of the outer $2N+1$ vertices.  The  conclusions of Proposition \ref{prop:markov-2N} above regarding when $\mu_x$ and $M$ are mutually singular also hold in this context. 
\end{rmk}

\begin{rmk}
  It is worth noting that, as observed in \cite{FGKPexcursions}, the infinite path space $\Lambda^\infty$ of a finite $k$-graph $\Lambda$  with vertex matrices $A_i,$ $i=1,\dots, k,$  is always naturally homeomorphic and Borel isomorphic to the infinite path space of { the Cuntz--Krieger algebra ${\mathcal O}_{A^J}$}.  Here $J \in (\Z_+)^k$ and $A^J$ denotes the matrix $A_1^{J_1} \cdots A_k^{J_k}$. Therefore one can transport to
$\Lambda^\infty$ any of the Markov measures {for Cuntz--Krieger algebras} that were constructed by S. Bezugyli and P. Jorgensen in \cite{bezuglyi-jorgensen}. However, these Markov measures
will not necessarily produce a $\Lambda$-semibranching function system on $\Lambda^\infty$, since {the homeomorphism of \cite{FGKPexcursions} gives no information
on the Radon--Nikodym derivatives with respect to the standard prefixing operators by edges, namely %\st{${\tau_{e}}$}
${\sigma_e}$, $e \in v\Lambda^{e_i}$.  }

\end{rmk}

\section{A separable faithful  representation of $C^*(\Lambda)$}
\label{sec:faithful-rep}

In Theorem~3.6 of \cite{FGKP}, the authors constructed a representation of $C^*(\Lambda)$ associated to the standard $\Lambda$-semibranching function system described in Remark~\ref{rmk:defn-standard-repn}, and proved that this representation is faithful if and only if $\Lambda$ is aperiodic.
In this section, we first construct in Proposition~\ref{prop:sc-faithful}  a separable representation for $C^*(\Lambda)$ for $\Lambda$ row-finite and strongly connected, which arises from a $\Lambda$-semibranching function system 
that is faithful even when $\Lambda$ is not necessarily aperiodic. Then in Proposition  \ref{pr:sep-faith}  we extend this result to $k$-graphs that are row-finite but not necessarily strongly connected.

{The underlying Hilbert space $\H_x$ of the representation of Proposition~\ref{prop:sc-faithful} is defined via an inductive limit, but we show in Proposition \ref{prop:sc-faithful-SBFS} that $\H_x \cong \ell^2(X)$ for a discrete measure space $X$ (with counting measure). }
 This perspective enables us to realize the representations of Proposition~\ref{prop:sc-faithful} and Proposition \ref{pr:sep-faith} as  $\Lambda$-semibranching representations.  Incidentally, the same arguments used in Proposition \ref{prop:sc-faithful-SBFS} also enable us to show in Proposition \ref{prop:sc-faithful-SBFS-Lambda} that the standard representation of $C^*(\Lambda)$ on $\ell^2(\Lambda^\infty)$ is a $\Lambda$-semibranching representation, although not a separable one.

Let $\Lambda$ be a strongly connected $k$-graph.  Fix $x \in \Lambda^\infty$  and write $x = x_1 x_2 x_3 \cdots $, where $d(x_i) = (1, 1, \ldots, 1)$ for all $i$.  Let $v_i = r(x_i)$.
For each $i$, write $F_i = \Lambda v_i$ for the set of all morphisms (i.e., finite paths) in $\Lambda$ with source $v_i$.  Then $\ell^2(F_i)$ has basis $\{\xi^i_\lambda: \lambda \in F_i\}$.  Define  $\rho_i \in B(\ell^2(F_i) , \ell^2(F_{i+1}))$ by $\rho_i(\xi^i_\lambda) = \xi^{i+1}_{\lambda x_i} \in \ell^2(F_{i+1})$, and form the inductive limit Hilbert space
\begin{equation}
\label{inductiveHilbert}
\H_x := \varinjlim (\ell^2(F_i), \rho_i) = \left(\bigsqcup_{i \in \N} \ell^2(F_i) \right) /\sim,
\end{equation}
where $\xi^i_\lambda \sim \xi^j_\mu$ (with $i \leq j$) iff $\mu = \lambda x_i x_{i+1} \cdots x_{j-1}$.
For a generator $\xi^i_\lambda$ of $\ell^2(F_i)$, we will denote its equivalence class in $\H_x$ by $[\xi^i_\lambda]$.

Observe that $\H_x$ is separable, because $F_i$ is countable for all $i$. 
Moreover, the same $\lambda \in \Lambda$ may appear in both $F_i$ and $F_j$ without having $[\xi^i_\lambda] = [\xi^j_\lambda]$, if the infinite path $x$ passes through the same vertex multiple times.

For any fixed $\lambda  \in \Lambda,$ we define an operator $T_\lambda \in B(\H_x)$ by 
\begin{equation}
T_\lambda [\xi^i_\mu] = \left\{ \begin{array}{cl}
[\xi^i_{\lambda \mu}], & s(\lambda) = r(\mu) \\
0, & \text{ else.}
\end{array} \right. \label{eq:sc-faithful}
\end{equation}

\begin{prop}\label{prop:faithful-repn}
Let $\Lambda$ be a row--finite, strongly connected $k$-graph and $x \in \Lambda^\infty$.  The operators $\{T_\lambda\}_{\lambda \in \Lambda}$ of Equation \eqref{eq:sc-faithful} define a faithful separable representation $\pi_x$ of $C^*(\Lambda)$ on $\H_x$.
\label{prop:sc-faithful}
\end{prop}

\begin{proof}
This proof was inspired by Section 3 of \cite{davidson-power-yang-dilation}.

We first check that the operators $T_\lambda$ are well-defined.  Recall, then,
that if $[\xi^i_\nu] = [\xi^j_\mu]\in \H_x$, then there exists $k \geq i,j$ such that $\mu x_j \cdots x_k = \nu x_i \cdots x_k$. Assuming $i \leq  j$, the factorization property then forces $\mu = \nu x_i \cdots x_{j-1}$ (if $i < j$; if $i=j$, we have $\mu = \nu$). In either case, $[\xi^j_{\lambda \nu}]  = [\xi^i_{\lambda \mu}]$, and hence $T_\lambda$ is well defined.

We now check that the operators $\{T_\lambda\}_{\lambda \in \Lambda}$ define a representation of $C^*(\Lambda)$.  To that end, observe that
\[\begin{split}
\langle T^*_\lambda [\xi_\mu^i] \mid [\xi_\nu^j]\rangle &=\langle [\xi_\mu^i] \mid T_\lambda [\xi_\nu^j]\rangle\\
&=\begin{cases}\langle [\xi_\mu^i]\mid [\xi_{\lambda\nu}^j]\rangle\quad\text{if $s(\lambda)=r(\nu)$}\\0\quad\quad\quad\text{otherwise}\end{cases}\\
&=\begin{cases} 1\quad\text{if $[\xi_\mu^i]=[\xi_{\lambda\nu}^j]$ and $s(\lambda)=r(\nu)$}\\ 0\quad\quad\quad\text{otherwise}.\end{cases}
\end{split}\]
Thus 
\begin{equation}\label{eq:T_adjoint}
T^*_\lambda[\xi_\mu^i]=\begin{cases} [\xi_\nu^j]\quad\quad\text{if $[\xi_\mu^i]=[\xi_{\lambda\nu}^j]$}\\
0\quad\quad\quad\text{otherwise}\end{cases}
\end{equation}

One checks immediately that for any $v, w \in \Lambda^0$, $T_v = T_v^* = T_v^2$  and  $T_v T_w = \delta_{v,w} T_v$. A similarly straightforward check shows that $T_\lambda^* T_\lambda = T_{s(\lambda)}$ and that $T_\lambda T_\mu = \delta_{s(\lambda), r(\mu)} T_{\lambda \mu}$.

It remains to check that for any $n \in \N^k, v \in \Lambda^0$, we have $\sum_{\lambda \in v\Lambda^n}T_\lambda T_\lambda^* = T_v$.  To that end, 
 fix $\lambda$ and $[\xi_\mu^i]$, and  compute 
\[\begin{split}
T_\lambda T^*_\lambda[\xi_\mu^i]&=\begin{cases} T_\lambda[\xi_\nu^j]\quad\quad\text{if $[\xi_\mu^i]=[\xi_{\lambda\nu}^j]$}\\ 0 \quad\quad\quad\text{otherwise}\end{cases}\\
&=\begin{cases} [\xi_{\lambda\nu}^j]\quad\quad\text{if $[\xi_\mu^i]=[\xi_{\lambda\nu}^j]$, $s(\lambda)=r(\nu)$}\\
0\quad\quad\quad\text{otherwise}\end{cases}\\
&=\begin{cases} [\xi_\mu^i]\quad\quad\text{if $[\xi_\mu^i]=[\xi_{\lambda\nu}^j]$ and $s(\lambda)=r(\nu)$}\\ 0\quad\quad\text{otherwise}\end{cases}
\end{split}\]
Now, fix $n \in \N^k$ and $v \in \Lambda^0$.  Observe that 
\[\begin{split}
\Big( \sum_{\lambda\in v \Lambda^n} T_\lambda T_\lambda^* \Big) [\xi^i_\mu]&=\begin{cases} \sum_{\lambda\in v \Lambda^n} [\xi^i_\mu]\quad\quad\text{if $[\xi_\mu^i]=[\xi_{\lambda\nu}^j]$ and $s(\lambda)=r(\nu)$}\\
0\quad\quad\text{otherwise.}\end{cases}\\
%&=\begin{cases} [\xi_\mu^i] \quad\quad\text{if $[\xi_\mu^i]=[\xi_{\lambda\nu}^j]$ and $s(\lambda)=r(\nu)$}\\
%0\quad\quad\text{otherwise}\end{cases}\\
\end{split}\]
If $r(\mu) = v$, then choose $k>i$ large enough so that $d(\mu x_i \cdots x_{k-1}) \geq n$.  Then, $[\xi^k_{\mu x_i \cdots x_{k-1}}] = [\xi^i_\mu]$, and the factorization property tells us we can write $\mu x_i \cdots x_{k-1} = \lambda \nu$ for a unique $\lambda \in v\Lambda^n$.  Thus,
\[\begin{split}
\Big( \sum_{\lambda\in v \Lambda^n} T_\lambda T_\lambda^* \Big) [\xi^i_\mu]&=\begin{cases}[\xi^i_\mu]\quad\quad\text{if } r(\mu) = v\\
0\quad\quad\text{otherwise.}\end{cases}\\
\end{split}\]
In other words, $\sum_{\lambda \in v\Lambda^n} T_\lambda T_\lambda^* = T_v$ as claimed.

It now follows that the operators $\{T_\lambda\}_{\lambda \in \Lambda}$ satisfy the Cuntz--Krieger relations, and thus generate a representation $\pi_x$ of $C^*(\Lambda)$ on $\H_x$.

We would like to use the gauge-invariant uniqueness theorem (Theorem 3.4 of \cite{KP}) to show that this representation is faithful.  We begin by checking that $T_v$ is nonzero for each $v \in \Lambda^0$.  To see this, fix $v \in \Lambda^0$.  Since $\Lambda$ is strongly connected, there exists $\lambda \in v\Lambda r(x_1)$.  We have $T_v [\xi^1_\lambda] = [\xi^1_\lambda]$; since $[\xi^1_\lambda]$ is a nontrivial element of $\H_x$, the operator $T_v$ is nonzero, as desired.

In order to apply the gauge-invariant uniqueness theorem, we must establish the existence of a gauge action on $\pi_x(C^*(\Lambda))$.  We do this by defining, for each $z \in \T^k$, a unitary $U_z \in B(\H_x)$:
\begin{equation}
\label{actionHilbertspace}
U_z[\xi^i_\mu] = z^{d(\mu) -(i, \ldots, i)} [\xi^i_\mu].
\end{equation}
It is a straightword calculation that $U_z$ is well defined and is unitary.  Thus we can define an action of $\T^k$ on $\pi_x(C^*(\Lambda))$ by $z \cdot T_\lambda := \Ad U_z (T_\lambda).$  
We check that $\pi_x \circ \gamma_z = \Ad U_z$ for any $z \in \T^k$, where $\gamma_z$ denotes the gauge action of $\T^k$ on $C^*(\Lambda)$.  The gauge invariant uniqueness theorem now tells us that $\pi_x$ is a faithful representation of $C^*(\Lambda)$ on the separable Hilbert space $\H_x$, as claimed.
\end{proof}

\begin{prop}
\label{prop:sc-faithful-SBFS}
Let $\Lambda$ be a strongly connected $k$-graph and fix $x \in \Lambda^\infty$.  The Hilbert space $\H_x$ is of the form $\ell^2(X)$ {with counting measure; for the definition of $X$ see Equation \eqref{def-space-count-measure}}. %{\color{blue}E would remove the mention of counting measure; that's implicit in the use of $\ell^2$.}
 Moreover, the faithful separable representation of Theorem \ref{prop:sc-faithful} is a $\Lambda$-semibranching representation.
\end{prop}
\begin{proof}
For each $i \geq 1$, define $G_i = \Lambda v_i \backslash \Lambda x_{i-1} \subseteq F_i.$  Equivalently,
%\footnote{Elizabeth: These definitions are actually not equivalent.  For example, suppose $\Lambda$ has only one vertex; then the displayed-equation definition gives you $G_1 = \Lambda$ and $G_i  = \emptyset$ for $i \geq 2$.  On the other hand, the crossed-out definition will usually also give you $G_i$ nonempty for larger $i$.  For example, suppose $\Lambda$ is a 2-graph with  one vertex and three edges $a, b,c$ where $a$ is color 1 and $b, c$ are color 2, and we have the factorization rule 
%\[ ab = ba, a c= ca.\]
%Then take $x = (ab) (ac) (ab) (ac) (ab)(ac) \cdots$.  The question is, where should the path $ac$ live?  The crossed-out definition would put it in $G_2$, and $G_4$, and $G_\ell$ for any even $\ell$.  But the displayed-equation definition would only put it in $G_1$.}
\[G_i = F_i \backslash \left( \bigcup_{j < i} F_j x_j \cdots x_{i-1} \right).\]
 Thus by definition of $\mathcal{H}_x$ we have $\mathcal{H}_x \supseteq \oplus_{i\ge 1} \ell^2(G_i)$. To see that $\mathcal{H}_x\subseteq \oplus _{i\ge 1} \ell^2(G_i)$, first note that any vector of $\mathcal{H}_x$ is of the form $[\xi^i_\mu]$, where $\mu\in F_i$ for some $i\ge 1$ by definition. If $\mu\in F_i \setminus  \left( \bigcup_{j < i} F_j x_j \cdots x_{i-1}\right)$, then $\xi^i_\mu\in \ell^2(G_i)$. If $\mu\in F_i$ and $\mu$ lies in  $\cup_{j<i}F_j x_j \cdots x_{i-1}$, then there exists a unique $\ell \le i$ and $\widetilde{\mu}\in F_\ell \setminus \cup_{j<\ell} F_j x_j \cdots x_{\ell -1}$ such that $\mu=\widetilde{\mu}x_\ell x_{\ell+1}\dots x_{i-1}$. So $\xi^i_\mu \sim \xi^{\ell}_{\widetilde{\mu}}$ and $\xi^{\ell}_{\widetilde{\mu}}\in \ell^2(G_\ell)$, and hence $\mathcal{H}_x\subseteq \oplus_{i\ge 1} \ell^2(G_i)$. 
Consequently,
\[ \H_x = \bigoplus_{i \geq 1} \ell^2(G_i).\]

Set 

\begin{equation}
\label{def-space-count-measure}
X : = \bigsqcup_{i \geq 1} G_i,
\end{equation} 
and let $m$ denote counting measure on $X$: 
\[ m\left( \nu \right) = 1 \ \forall \ \nu \in G_i.\]
Then $\H_x = \bigoplus_{i\geq 1} \ell^2(G_i) = \ell^2(X) = L^2(X, m).$ 

We now describe the $\Lambda$-semibranching function system on $(X, m)$ which gives rise to the representation $\{T_\lambda\}_{\lambda \in \Lambda}$.
For a vertex $v \in \Lambda^0$, define $D_v = \{ \nu \in \bigsqcup_{i \geq 1} G_i: r(\nu) = v\},$ and for $\lambda \in \Lambda$ set
$\tau_\lambda: D_{s(\lambda)} \to D_{r(\lambda)}$ by 
\[ \tau_\lambda(\nu) = \rho, \text{ where } \rho \in G_j \text{ and } \lambda \nu = \rho  x_j \cdots x_{i-1} \text{ if } \nu \in G_i.\]
To see that $\tau_\lambda$ is well-defined, fix $\nu\in G_i$ and suppose that there exist $j_1\ne j_2 \le i$ and $\rho_1\in G_{j_1}$, $\rho_2\in G_{j_2}$ such that
\[
\lambda \nu=\rho_1 x_{j_1} x_{j_1+1}\dots x_{i-1} = \rho_2 x_{j_2} x_{j_2+1}\dots x_{i-1}
\]
Then by the factorization property, assuming without loss of generality that $j_2\ge j_1$, we must have $\rho_1 x_{j_1}\dots x_{j_2-1}=\rho_2$. Thus there exists a unique $j$ and $\rho\in G_{j}$ such that $\tau_\lambda(\nu)=\rho$, and hence $\tau_\lambda$ is well-defined.

It follows that
\[ R_\lambda = Ran(\tau_\lambda) = \{ \rho: \rho \in G_j \text{ for some } j \text{ and } \rho x_j \cdots x_i (0, d(\lambda)) = \lambda \text{ for some } i\}.\]
 If $ d(\lambda)=n$, then for $\rho\in G_j \cap R_\lambda$ find the smallest $i\geq j$ such that
 \[
 d(\rho)+(i-j)(1,\dots,1)\ge n.
 \]
 Then define the coding map $\tau^n$ on $G_j \cap R_\lambda$ by\footnote{If $i=j$ then we take $\tau^n(\rho) = \rho(n, d(\rho))$.}
 \begin{equation}
\label{eq:faithful-tau-n}
\tau^n(\rho) = \rho x_j \cdots x_{i-1}(n, d(\rho) + (i-j)(1, \ldots, 1)) \in G_i.
\end{equation}

Now it is straightforward to see that
\[ \tau^n \circ \tau_\lambda (\nu) = \nu,\]
justifying the name ``coding map.''

We claim that the sets and maps described above satisfy Conditions (a) - (d) of Definition \ref{def-lambda-SBFS-1} and hence define a $\Lambda$-semibranching function system on $X$.

First, we fix $n \in \N^k$ and check that for each $\nu \in \bigsqcup_{i\geq 1} G_i$ we have $\nu \in R_\lambda$ for precisely one $\lambda \in \Lambda^n$, which implies that $X = \bigsqcup_{\lambda \in \Lambda^n} R_\lambda$ for any $n\in \N^k$. Given $\nu \in G_i$, let $j \geq i$ be the smallest integer such that $d(\nu) + (j-i)(1, \ldots, 1) \geq n$.  Set $\lambda = \nu x_i \cdots x_{j-1}(0, n)$; then $\nu \in R_\lambda$. Moreover, for any other $\lambda' \in \Lambda^n$, we have $\nu x_i \cdots x_{j-1}(0, n) \not= \lambda'$, so $\nu \in R_\lambda$ for a unique $\lambda \in \Lambda^n$.  Since we are working in a discrete measure space, the Radon-Nikodym derivatives of the prefixing maps $\tau_\lambda$ are constantly equal to 1 on $D_{s(\lambda)}$.  This completes the check of Condition (a) of Definition \ref{def-lambda-SBFS-1}. 

By our   hypothesis that $\Lambda$ is strongly connected,  if $v \not= v_i$ for any $i$, we have  $\emptyset \not= v \Lambda v_i \subseteq F_i$ for all $i$.  
This implies the existence of at least one $\nu \in v \Lambda \cap \bigsqcup_{j \geq 1} G_j$, so $m(D_v) > 0$.  On the other hand, if $v=v_i$ then $v_i \in G_i$ is an element of $D_{v_i}$. Again, we have $m(D_{v_i}) > 0$, so  Condition (b) is satisfied.

The description in Equation \eqref{eq:faithful-tau-n} of the coding map $\tau^n$ makes it easy to check that, for $\rho \in G_j$ and for any $m, n\in \N^k$,
\[ \tau^{m+n}(\rho) = \rho x_j \cdots x_{\ell-1}(m+n, d(\rho) + (\ell-j)(1, \ldots, 1)) = \tau^m \circ \tau^n(\rho),\]
 where $\ell$ is the smallest such that $d(\rho)+(\ell-j)(1,\dots, 1)\ge m+n$.
In other words, Condition (d) holds.  

Similarly, if
$\tau_\lambda \circ \tau_\nu(\rho) = \alpha \in G_\ell$ for some $\rho \in G_j$, then there exist $j \geq i \geq \ell $ and $\eta \in G_i$ with $\alpha x_\ell \cdots x_{i-1} = \lambda \eta$ and $\eta x_i \cdots x_{j-1} = \nu \rho$.  On the other hand, if $\tau_{\lambda \nu}(\rho) = \beta \in G_n$, then 
\[ \beta x_n \cdots x_{j-1} = \lambda \nu \rho = \lambda \eta x_i \cdots x_{j-1} = \alpha x_\ell \cdots x_{j-1}.\]
Since $\alpha$ and $\beta$ are both in $\bigsqcup_{i\geq 1} G_i$, the factorization rule now implies that $n = \ell $ and $\alpha = \beta$.  
% first suppose, wlg, n is smaller than \ell -- then \beta x_n \cdots x_{\ell-1} = \alpha, which contradicts the fact that apha is in G_\ell.
It follows that Condition (c) of Definition \ref{def-lambda-SBFS-1} is also satisfied, so the sets $D_v, R_\lambda$ with the coding and prefixing maps $\tau_\lambda, \tau^n$ determine a $\Lambda$-semibranching function system.

Since $(X, m)$ is a discrete measure space, the representation $\{S_\lambda\}_{\lambda\in \Lambda}$ of $C^*(\Lambda)$ given by this $\Lambda$-semibranching function system, described in Theorem \ref{thm:SBFS-repn} above, has the following formula.  Given $\eta \in G_i$, write $\delta_\eta \in L^2(X, m)$ for the indicator function supported at $\eta$. For $\nu \in G_j$, we have 
\begin{align*}
 S_\lambda (\delta_\eta)(\nu) &= \begin{cases}
  0, & \nu \not \in R_\lambda \\
  \delta_\eta(\tau^{d(\lambda)}(\nu)), & \text{ else.}
 \end{cases} 
% &= \begin{cases}
% 0, & \nu \not\in R_\lambda \\
% \delta_\eta(\nu x_j \cdots x_{\ell-1}(d(\lambda), d(\nu) + (\ell-j)(1, \ldots, 1))), & \nu \in R_\lambda \text{ and } \ell \text{ the smallest s.t.} \\
% & d(\nu) + (\ell-j)(1, \ldots, 1) \geq d(\lambda).
% \end{cases}
\end{align*}
By construction, we have $\tau^{d(\lambda)}(\nu) = \nu x_j \cdots x_{\ell-1}(d(\lambda), d(\nu) + (\ell-j)(1, \ldots, 1)) \in G_\ell$.  Thus, the above formula becomes 
\begin{align*} S_\lambda(\delta_\eta)(\nu) &= \begin{cases} 
0, & \nu \not\in R_\lambda \text{ or } \tau^{d(\lambda)}(\nu) \not\in G_i \\
\delta_\eta(\tau^{d(\lambda)}(\nu))), & \text{ else.}
\end{cases}\\
& =\begin{cases} 
1, & \lambda \eta = \nu x_j \cdots x_{i-1} \\
0, & \text{ else.}
\end{cases}
\end{align*}
Since $\lambda \eta = \nu x_j \cdots x_{i-1}$ iff $\nu = \tau_\lambda(\eta)$, we can rewrite this as  
\begin{equation}
\label{eq:faithful-rep-SBFS-formula}
S_\lambda (\delta_\eta) = \delta_{\tau_\lambda(\eta)}.\end{equation}

To finish the proof, we observe that, under the isomorphism $\H_x \cong L^2(X, m)$, Equation \eqref{eq:faithful-rep-SBFS-formula} agrees with the formula for $T_\lambda$ given in Equation \eqref{eq:sc-faithful}.  This follows from the observation that $\tau_\lambda(\eta) \sim \lambda \eta$ by construction, so $[\xi^j_{\tau_\lambda(\eta)}] = [\xi^i_{\lambda \eta}]$.

\end{proof}

We now study under what conditions these representations are unitarily equivalent to one another.

\begin{prop}
Let $\Lambda$ be a row--finite strongly-connected $k$-graph.
 If $x,y \in \Lambda^\infty$ are infinite paths such that $\sigma^m(x) = \sigma^n(y)$ for some $m, n \in \N^k$, then the corresponding representations $\pi_x$ and $\pi_y$ given in Theorem~\ref{prop:faithful-repn} are unitarily equivalent.
\end{prop}
\begin{proof}
Suppose that there are infinite paths $x,y$ such that $\sigma^m(x)=\sigma^n(y)$ for some $m,n\in \N^k$.
We write $x=x_0x_1\cdots$ and $y=y_0y_1\cdots$, where $d(x_i)=d(y_i)=(1,1,\cdots,1)$ for all $i$.
Recall that in this setting, we have $x(\vec{k},\vec{\ell})=x_k x_{k+1}\cdots x_\ell$ for $\vec{k}=(k,k\cdots,k) \leq \vec{\ell}=(\ell,\ell,\cdots,\ell)$

To construct an isomorphism $\phi: \H_x \to \H_y$,  fix $[\xi^i_\mu] \in \H_x$. 
Without loss of generality, assume $\vec{i}\ge m$.  Then $\sigma^n(y)=\sigma^m(x)$ implies that $\sigma^{n-m+\vec{i}}(y)=\sigma^{\vec{i}}(x)$, and hence $y(n-m+\vec{i},\infty)=x(\vec{i},\infty)$. Choose the minimum $j\in \N$ such that $\vec{j}\ge n$ and $\vec{j}-n \ge \vec{i}-m$. Then let 
\[
\lambda_{i,j}=y(n-m+\vec{i},\vec{j}).
\]
Note that, if we write $q = \vec{j} - n - (\vec{i} - m)\in \N$, then $\lambda_{i,j}=x(\vec{i},q)$.
Thus, $\lambda_{i,j}$ is the common segment of $x$ and $y$ that lies between the vertices $r(x_i)$ and $r(y_j) $. It follows that multiplying by $\lambda_{i,j}$ on the right takes  $F_{i,x}$ to $F_{j,y}$. 

To be precise, we define
\begin{equation}
\phi([\xi^i_\mu]_x) : = [\xi^j_{\mu \lambda_{i,j}}]_y.\label{eq:phi-1}
\end{equation}

%If $\vec{i} \not \geq m$, choose $\ell \in \N$ such that $[\xi^i_\mu]_x = [\xi^\ell_{\mu x_i \cdots x_{\ell-1}}]_x$ and $\vec{\ell} \geq m$, and define $\phi$ by
%\begin{equation}
%\phi([\xi^i_\mu]_x) := \phi([\xi^\ell_{\mu x_i \cdots x_{\ell - 1}}]_x) = [\xi^j_{\mu x_i \cdots x_{\ell -1} \lambda_{\ell, j}}]_y.
%\label{eq:phi-2}
%\end{equation}

We first verify that $\phi$ is well defined: suppose that $[\xi^i_\mu]_x = [\xi^k_{\mu \mu'}]_x$, where $\vec{k} > \vec{i} \geq m$. We then have $\mu' = x_i \cdots x_{k-1}$ and 
\[\phi([\xi^i_\mu]_x) = [\xi^j_{\mu \lambda_{i,j}}]_y \text{ and } \phi([\xi^k_{\mu \mu'}]_x) = [\xi^\ell_{\mu \mu' \lambda_{k, \ell}}]_y,\]
where $j$ is the coordinatewise maximum of $n-m + \vec{i}$ and $\ell$ is the coordinatewise maximum of $n-m +\vec{k}$.

 This definition of $j, \ell$ implies that  $\ell - k = j-i$ is the coordinatewise maximum of $n-m$; consequently, 
 \[ d(\lambda_{i,j}) =\vec{j} - n + m - \vec{i} = \vec{\ell} -n+m - \vec{k} = d(\lambda_{k, \ell}).\]
  Also, $\ell - j = k -i > 0$, so we can write 
\[[\xi^j_{\mu \lambda_{i,j}}]_y = [\xi^\ell_{\mu \lambda_{i,j} \eta}]_y\]
where $\eta = y(\vec{j}, \vec{\ell})$.  In other words, $d(\eta) = \vec{\ell} - \vec{j} = \vec{k} - \vec{i} = d(\mu').$

Since $\vec{i} \geq m$, the finite paths $\mu' \lambda_{k,\ell}$ and $\lambda_{i,j} \eta$ lie on both $x$ and $y$.  In fact, 
\[s(\mu' \lambda_{k, \ell}) = r(y_\ell)= s(\lambda_{i,j} \eta) \text{ and } r(\mu' \lambda_{k,\ell}) = r(x_i) = r(\lambda_{i,j} \eta).\]
Moreover, $d(\mu' \lambda_{k,\ell}) = d(\mu') + d(\lambda_{k,\ell}) = d(\eta) + d(\lambda_{i,j}).$  The factorization property then tells us that \[\mu' \lambda_{k, \ell} = \lambda_{i,j} \eta.\]  

It now follows that 
\[\phi([\xi^i_\mu]_x) = [\xi^j_{\mu \lambda_{i,j}}]_y = [\xi^\ell_{\mu \lambda_{i,j}\eta}]_y = [\xi^\ell_{\mu \mu' \lambda_{k, \ell}}]_y = \phi([\xi^k_{\mu\mu'}]_x),\]
so $\phi$ is well defined.

To see that $\phi$ is surjective, fix $\nu \in( F_j)_y$ and consider the associated element $[\xi^j_\nu]_y \in \H_y$.  Pick $t \geq j$ large enough to ensure the existence of $\ell \in \N$ with $m \leq \vec{\ell} \leq  m + \vec{t} - n $: in other words, $\vec{t} - n \geq (\max_m - \min_m) \cdot (1, \ldots, 1)$. Since $\sigma^{\vec{t}}(y) = \sigma^{m+\vec{t} - n}(x)$, our choice of $t$ and $\ell$ ensure that $\lambda_{\ell, t}$ is a sub-path of $\nu y_j \cdots y_{t-1}$.  We can therefore write 
\[ \nu y_j \cdots y_{t-1} = \tilde{\nu} \lambda_{\ell, t}\]
for some $\tilde{\nu} \in (F_\ell)_x$.  It follows that $[\xi^j_\nu]_y= \phi([\xi^\ell_{\tilde{\nu}}]_x)$, so $\phi$ is surjective as claimed.

To see that $\phi$ is injective, suppose that $\phi([\xi^i_\mu]_x) = \phi([\xi^\ell_\nu]_x)$.  Without loss of generality, suppose that $\vec{i}\geq \vec{\ell}  \geq m$, so that 
\[\phi([\xi^i_\mu]_x) = [\xi^j_{\mu \lambda_{i,j}}]_y \quad \text{ and } \quad \phi([\xi^\ell_\nu]_x) = [\xi^{h}_{\nu \lambda_{\ell, h}}]_y \]
where $h$ is the coordinatewise maximum of $n-m+\vec{\ell}$ and $j$ is the coordinatewise maximum of $n-m+\vec{i}$.

Since $i \geq \ell$, we can write $i = \ell + q$ for $q \in \N$. Consequently, the coordinatewise maximum $j$ of $n-m+\vec{i}$ is the same as the sum of the coordinatewise maximum of $n-m+\vec{\ell}$ and $q$. In other words, $j=h+q$.
It follows that 
\[d(\lambda_{i,j}) = \vec{j} - n+m -\vec{i} = \vec{q} + \vec{h} - n +m - \vec{i} = \vec{h} - n +m - \vec{\ell} = d(\lambda_{\ell, h}).\]

Since  $j \geq h$, the equivalence relation on $\H_y$  implies that 
\[\nu \lambda_{\ell, h}y_h \cdots y_{j-1}  = \mu \lambda_{i,j}  \text{ if $j>h$; }\]
if $h=j$ then $i=\ell$ and we must have $\nu = \mu$. 

Observe that $j-h = i-\ell = q$.  Assuming that $j>h$, we can write 
\[\nu \lambda_{\ell, h} y_h \cdots y_{j-1} = \nu x_\ell \cdots x_{\ell + j-h-1} \lambda_{\ell +j-h, j} = \nu x_\ell \cdots x_{i-1} \lambda_{i,j}.\]

Then the factorization property implies that 
\[\nu x_\ell \cdots x_{i-1} = \mu,\]
and consequently $[\xi^\ell_\nu]_x = [\xi^i_\mu]_x$.  In other words, $\phi$ is injective.

Finally it is straightforward to check that $\phi\circ T^x_\lambda=T^y_\lambda\circ\phi$ for $\lambda\in\Lambda$, and hence the representations $\pi_x, \pi_y$ are equivalent, as claimed.
\end{proof}
%{\color{purple} Question: $\pi_x = \pi_y$ implies $x, y$ are shift equivalent. Can we prove this under some suitable hypotheses?}

Observe that the representation $\pi_x$ of Proposition~\ref{prop:sc-faithful} is in fact well-defined for any row-finite %{\color{teal} strongly connected} {\color{blue}E: the teal above is not consistent with the rest of the paragraph, or with the following theorem} 
source-free  higher-rank graph $\Lambda$ and any $x \in \Lambda^\infty$, even if $\Lambda$ is not strongly connected. We only required the hypothesis that $\Lambda$ be strongly connected in order to ensure that $T_v$ was nonzero for each $v$.  However, a similar construction will give us a separable faithful representation of $C^*(\Lambda)$ for any row-finite, source-free $k$-graph $\Lambda$.
\begin{thm}
Let $\Lambda$ be a row-finite source-free $k$-graph.  There is a faithful separable representation of $C^*(\Lambda)$.
\label{pr:sep-faith}
\end{thm}
\begin{proof}
For each vertex $v \in \Lambda^0$, choose an infinite path $y_v$ with %$r(y) = v$ {\color{teal} 
$r(y_v) = v$. (The fact that $\Lambda$ is source-free implies we can always do this.) Since $\Lambda$ is a countable category, $\Lambda^0$ is countable, and we have made countably many choices.  Define
\[
 \H := \bigoplus_{v \in \Lambda^0} \H_{y_v}, \quad \pi := \bigoplus_{v\in \Lambda^0} \pi_{y_v},
\]
where $\H_{y_v}$ is the Hilbert space defined in Equation \eqref{inductiveHilbert}, and $\pi_{y_v}$ is the representation defined in Equation \eqref{eq:sc-faithful}.
Then $\H$ is a separable Hilbert space and $\pi$ is a representation of $C^*(\Lambda)$ on $\H$.  We know that $\pi(t_\mu)$ is nonzero for each $\mu \in \Lambda$, because
\[\pi_{y_{s(\mu)}}(t_\mu) [\xi^1_{s(\mu)}] = [\xi^1_\mu]\]
is a nonzero generator of $\H_{y_{s(\mu)}}$ and hence of $\H$.

Moreover, the unitary action {$\gamma$} of $\T^k$ on $\H_{y_v}$ defined in Equation \eqref{actionHilbertspace}  extends to a unitary action of $\T^k$ on $\H$ via the diagonal action.  Similarly, the fact that each representation $\pi_{y_v}$ intertwines this action with the gauge action on $C^*(\Lambda)$ implies that we have
\[z \cdot \pi(T) = \pi (\gamma_z(T)),\]
so  Theorem 2.1 of \cite{bprz} (the gauge-invariant uniqueness theorem) tells us that $\pi$ is a faithful separable representation of $C^*(\Lambda)$.
\end{proof}

Often, the trickiest part in checking that a family of subsets and coding/prefixing maps constitutes a $\Lambda$-semibranching function system is computing the Radon-Nikodym derivatives.  On a discrete measure space, this computation is rendered trivial, as we saw above.  Thus, in the spirit of Proposition \ref{prop:sc-faithful-SBFS}, we also have the following:

\begin{prop}\label{prop:sc-faithful-SBFS-Lambda}
Let $\Lambda$ be a finite, strongly connected $k$-graph.  Let $\pi_T$ be the standard infinite path representation of $C^*(\Lambda)$ on $\ell^2(\Lambda^\infty)$, given by 
\begin{equation}\label{eq:std-rep-on-ell-2}
\pi_T(t_\lambda)=\begin{cases} \xi_{\lambda x} \quad \text{if}\;\; r(x)=s(\lambda) \\ 0 \quad\quad\text{otherwise} \end{cases}
\end{equation}
where $x\in\Lambda^\infty$ and $\xi_x$ denote the associated basis vector of $\ell^2(\Lambda^\infty)$. Then $\pi_T$ is a $\Lambda$-semibranching representation.
\end{prop}
\begin{proof}
We first define subsets $\{D_v\}_{v\in \Lambda^0}$ of $\Lambda^\infty$ and prefixing and coding maps $\tau_\lambda, \tau^n$ which give rise to a $\Lambda$-semibranching function system on $\Lambda^\infty$.  Namely, we have 
\[ D_v = v\Lambda^\infty, \qquad \tau_\lambda(x) = \lambda x, \qquad \tau^n(x) = \sigma^n(x).\]
The fact that Condition (a) of Definition \ref{def-lambda-SBFS-1} holds for these sets follows from the fact that, for fixed $n \in \N^k$, every infinite path $x$ is of the form $\lambda y$ for a unique $\lambda \in \Lambda^n$.  Since $\Lambda^\infty$, in this setting, is a discrete measure space, the Radon-Nikodym derivatives are again  constantly equal to 1, and moreover Condition (b) holds.   Conditions (c) and (d) are immediate consequences of the factorization property.

Thus, the sets $\{D_v\}_{v\in \Lambda^0}$, together with the prefixing and coding maps $\{\tau_\lambda, \tau^n: \lambda \in \Lambda, n \in \N^k\}$, constitute a  $\Lambda$-semibranching function system on $\Lambda^\infty$, viewed as a discrete measure space.  The associated representation $\{S_\lambda\}_{\lambda \in \Lambda}$ is given by (for $x, y \in \Lambda^\infty$)
\[ S_\lambda(\delta_y)(x) = \chi_{Z(\lambda)}(x) \delta_y(\sigma^{d(\lambda)}(x)) = \delta_{\lambda y}(x),\]
so $S_\lambda(\delta_y) = \delta_{\lambda y}$ agrees with the formula for the standard infinite path representation \eqref{eq:std-rep-on-ell-2}.
\end{proof}

%%%%%%%%%%%%%%%%%%%%%%%%%%%%%%%%%%%%%%%%%%%%%%%   

\end{document}